\documentclass[11pt]{article}

\usepackage{comment}
\usepackage[utf8]{inputenc}
\usepackage[english]{babel}
\usepackage{amsmath,amsthm,amsfonts}
\usepackage{graphicx}
\usepackage{bm}
\usepackage{tikz}
\usepackage{multirow}
\usepackage[makeroom]{cancel}
\usepackage{subcaption}

\newtheorem{theorem}{Theorem}[section]

\newtheorem{lemma}[theorem]{Lemma}

\graphicspath{ {./figs/}}

\numberwithin{equation}{section}

\renewcommand{\div}{\mathop{\rm div}\nolimits}

\newcommand{\xoe}{{x\over\epsilon}}
\newcommand{\ep}{\epsilon}
\newcommand{\jrp}[1]{{\color{brown}{#1}}}

\newcommand{\beq}{\begin{equation}}
\newcommand{\eeq}{\end{equation}}
\newcommand{\beqas}{\begin{eqnarray*}}
\newcommand{\eeqas}{\end{eqnarray*}}
\newcommand{\beqq}{\begin{equation*}}
\newcommand{\eeqq}{\end{equation*}}
\newcommand{\bsp}{\begin{split}}
\newcommand{\esp}{\end{split}}
\newcommand{\eps}{\epsilon}
\newcommand{\norm}[1]{\left\lVert#1\right\rVert}

\newcommand{\IR}{\mathbb{R}}

\clearpage
\title{\bf Homogenization of a multiscale multi-continuum system}

\author{ 
Jun Sur Richard Park,\\
Department of Mathematics,\\
Texas A\&M University, College Station, TX 77843 \\[10pt]
Viet Ha Hoang,\\
Division of Mathematical Sciences,\\
School of Physical and Mathematical Sciences,\\
 Nanyang Technological University, Singapore 637371 
}
\begin{document}
\maketitle
\begin{abstract}
%Direct numerical simulation on multi-continuum systems is challenging and sometimes not feasible due to multiple scale or high contrast.
%For this reason, one needs to construct effective coefficients and homogenized equations of the given original systems. 
We study homogenization of a locally periodic two-scale dual-continuum system where each continuum interacts with the other. 
%Due to special interaction terms of given system, 
Equations for each continuum are written separately with interaction terms (exchange terms) added. The homogenization limit depends strongly on the scale of this continuum interaction term with respect to the microscopic scale. In J. S. R. Park and V. H. Hoang, {\it Hierarchical multiscale finite element method for multicontinuum media}, arXiv:1906.04635, we study in details the case where the interaction terms are scaled as $O(1/\ep^2)$ where $\ep$ is the microscale of the problem. We establish rigorously homogenization limit for this case where we show that in the homogenization limit, the dual-continuum structure disappears. In this paper, we consider the case where this term is scaled as $O(1/\ep)$. This case is far more interesting and difficult as the homogenized problem is a dual-continuum system which contains features that are not in the original two scale problem. 
%Homogenization of our dual-continuum system is new because given coupled interaction terms are scaled as the inverse of micro-scale ($\mathcal{O}(\frac{1}{\ep}$)).
%We first derive cell problems and the homogenized equations using two-scale asymptotic expansions of the solutions to the original equations. Among various choices of the scale of the interaction terms, the $\frac{1}{\ep}$ scale gives the most interesting homogenized limit. 
In particular, the homogenized dual-continuum system contains extra convection terms and negative interaction  coefficients while the interaction coefficient between the continua in the original two scale system obtains both positive and negative values. We prove rigorously the homogenization convergence. We also derive rigorously a homogenization convergence rate. Homogenization of dual-continuum system of this type has not been considered before.
% a homogenization error estimates for our problem. These justifications of homogenization for our dual-continuum system is difficult and new compared to others (for example, see \cite{park2019hierarchical}) due to the complicated form of the homogenized equations. 
% We also present the proofs of the existence and uniqueness of the solutions to both original and homogenized equations. 
\end{abstract}
\section{Introduction}

In real life applications, media with multiple continua often involve multiple scales due to heterogeneous media property and complicated configuration of the continua. Simulations in those media are often very expensive and require some type of model reduction. One of the model reduction methods is multi-continuum approach \cite{barenblatt1960basic,warren1963behavior,kazemi1976numerical,wu1988multiple,pruess1982fluid} where equations for each continuum are written separately with some interaction terms (exchange terms) that represent interrelations between the continua. Those interaction terms are coupled, hence, one has to deal with a system with several coupled equations. In this paper, we study %\jrp{the} 
homogenization of a dual-continuum system with two-scale coefficients that are periodic with respect to micro-scale variable. 

%Some multiscale methods solve given equations without performing homogenization. 
There has been much effort to develop numerical methods for solving {multiscale system with reduced complexity. We mention exemplarily the multiscale finite element methods (MsFEM) \cite{eh09,ehg04}, the generalized multiscale finite element method (GMsFEM)\cite{chung2016adaptive},  the heterogeneous multiscale methods (HMM) \cite{ee03},  and the local orthogonal decomposition (LOD) \cite{maalqvist2014localization}. 
For multi-continuum systems, numerical methods such as the GMsFEM (\cite{chung2017coupling}), constraint energy minimizing (CEM) (\cite{cheung2018constraint}) and non-local multi-continuum method (NLMC)(\cite{vasilyeva2019nonlocal,chung2018non,vasilyeva2018three}) have been developed and employed.}
% basis functions that capture microscopic information are obtained by solving local problems in coarse blocks. Generalized multiscale finite element method (GMsFEM) \cite{chung2016adaptive, egh12}, follows the concept of MsFEM but add more basis functions in each coarse neighborhood. 
%They solve local spectral problems in snapshot spaces and the dominant eigenvectors corresponding to small eigenvalues form the basis space. 
%Constraint Energy Minimizing Generalized Multiscale Finite Element Methods (CEM-GMsFEM) \cite{chung2018constraint} follows the outline of GMsFEM to construct auxiliary basis functions. Using these functions, one obtains basis functions by solving energy minimization problems on subdomains. 
These methods are known to have several advantages and good convergence results. However, they do not use the periodicity or local periodicity of the coefficients and sometimes become expensive when the fine mesh size is necessarily
much smaller than the coarse mesh size. 
%In this paper, we will perform an upscaling method, homogenization, taking into account the microscopic periodicity of given two-scale coefficients. 
When the coefficients of the multiscale problem are periodic or locally periodic, the multiscale equations can be approximated by the equivalent homogenized equations whose coefficients do not vary rapidly. The theory of homogenization has a long and {successful} history. We mention only those now classical references Bensoussan et al. \cite{papanicolau1978asymptotic}, Bakhvalov and Panasenko \cite{bakhvalov1989homogenisation} and Jikov et al. \cite{jikov2012homogenization}. However, for multiscale multi continuum systems where the multiple continua interact with each other, there has been very little literature. As we show in this paper, homogenization of these systems can result in very interesting effective phenomena that are not often seen in homogenization literature.
%We derive effective properties and homogenized equations by solving local representative volume element (RVE) problems in each coarse block.
%There have been less attempts to analytically homogenize the multi-continuum systems. 
In \cite{park2019hierarchical}, we study homogenization of the two-scale dual-continuum system \eqref{eq:main1} with the interaction terms (exchange terms) being scaled as ${\mathcal O}(\frac{1}{\ep^2})$ instead, where $\ep$ represents the microscopic scale. We prove that in the homogenization limit, the two-continuum feature disappears, i.e.  $u_1^\ep$ and $u_2^\ep$  converge to the same limit.
In this paper, we consider the case
%perform homogenization of a different two-scale dual-continuum system (\ref{eq:main1}) 
where the interaction terms are scaled as ${\mathcal O}(\frac{1}{\ep})$. We prove that the homogenization system for this case is far more interesting and complicated. The homogenized two-continuum system consists of convection continuum interacting terms which do not appear in the original two-scale system \eqref{eq:main1}. 
% The homogenization of  multi-continuum systems highly depends on how the interaction terms are scaled in terms of $\ep$. We perform two-scale asymptotic expansions \cite{papanicolau1978asymptotic, bakhvalov1989homogenisation, jikov2012homogenization} and $\frac{1}{\ep}$ scale gives the most interesting homogenization limit as the resulting homogenized equations are complicated and not of the same form as original two-scale equations. 
%The effective equations have convection terms and 
Furthermore, the homogenized two-continuum system has negative interaction coefficients while the interaction coefficients of the original two scale system obtain both positive and negative values.  

%Studying homogenization of our system is challenging due to the complicated form of the homogenized equations. 
We provide a rigorous proof of homogenization convergence. The proof is significantly more difficult than that for the system in \cite{park2019hierarchical} because of the $\frac{1}{\ep}$-scale of the interaction terms in our system and the complicated homogenization limit.
We also derive a homogenization error under regularity conditions for the solutions to the cell problems and the homogenized equation.  Such a homogenization error has never been derived for multiscale multi-continuum systems before.
%We show that it is standard $\mathcal{O}(\ep^{\frac{1}{2}})$ rate where $\ep$ is the microscopic scale. 
The main results are Theorems \ref{weakconv} and \ref{hom_error}.
%The proof of homogenization error estimates for multi-continuum system is new. We have a system of two equations and main difficulty in the proof arises with the coupled unique exchange terms in the original equation. 

The paper is organized as follows. In Section 2, we set up the two scale multi-continuum systems. We perform the two scale asymptotic expansion to derive the homogenized multicontinuum systems. It is clear from the two scale asymptotic expansion that the homogenized system contains extra convection terms that do not appear in the original two scale system. We then state the main results of the paper on the convergence of the solution of the multiscale multi-continuum system to the solution of the homogenized multi-continuum system.
%
%we provide homogenization of our multi-continuum system. We derive cell problems and construct the homogenized equation. 
In Section 3, we present the proof of the convergence of the solutions of the multiscale system to the solution of the homogenized  system. 
%We show that the solutions to the original equations converges weakly to the solutions of the homogenized equations. 
In Section 4, we derive a corrector and prove a homogenization error estimate where the solutions to the cell problems and the solutions to the homogenized multi-continuum system are sufficiently regular. Finally, the appendix in the end of the paper contains the proofs of the existence and uniqueness of solutions to both the original two scale system and the homogenized equations. 

In this paper, we denote the gradient with respect to $x$ %by $\nabla$ when 
of a function that only depends on the variable $x$, or the variables $x$ and $t$ by $\nabla$.
By $\nabla_x$, we denote the partial gradient with respect to $x$ of a function that depends on $x$, $y$ and $t$. Repeated indices indicate summation. The notation $\#$ denotes spaces of periodic functions.

%%%%%%%%%%%%%%%%%%%%%%%%%%%% problem formulation %%%%%%%%%%%%%%%%%%%%%%%%%%%%%%%
\section{Problem formulation}
 \subsection{Two scale multicontinuum problem}
Let $\Omega$ be a bounded domain in $\mathbb{R}^d$. Let $Y$ be a unit cube in $\mathbb{R}^d$.
Let $Q(x,y)$, ${\mathcal C}_{ii}(x,y)$ and $\kappa_i(x,y)$ ($i=1,2$) be continuous functions on $\Omega\times Y$ which are $Y$-periodic with respect to $y$. We assume that 
\beq
\int_YQ(x,y)dy=0.
\label{eq:Qaverage}
\eeq
 Let $T>0$. Let $q$ be a function in $L^2((0,T)\times\Omega)$. 
%We require $\int_Y Q(x,y)  dy = 0$.
Let $\epsilon > 0$ be a small quantity that represents the microscopic scale the coefficients depend on. 
We define the two scale coefficients as
 \begin{equation}
\label{eq:main2}
 \begin{split}
{\mathcal C}_{ii}^\epsilon(x) = {\mathcal C}_{ii}(x,\frac{x}{\epsilon}),\ \kappa_i^\epsilon(x) = \kappa_i(x,\frac{x}{\epsilon}), \ \ i=1,2,\ \ \mbox{and}\ 
\ Q^\epsilon(x) = Q(x,\frac{x}{\epsilon}).
 \end{split}
 \end{equation}
%where ${\mathcal C}_{ii}$, $\kappa_i$ and $Q$ are $Y$-periodic functions from $\Omega \times Y$.
Let $H$ denote the space $L^2(\Omega)$ and $V$ denote the space $H^1_0(\Omega)$.  
We consider the following dual-continuum system.
%%%%%%%%%%%%%%%%%%%%%%%%%%%%%%.    fine-scale operator    %%%%%%%%%%%%%%%%%%%%%%%%%%%%%%%%%%%
\begin{equation}
\label{eq:main1}
 \begin{split}
{\mathcal C}_{11}^\epsilon(x){\partial u_1^\epsilon(t,x)  \over \partial t}=\text{div}(\kappa_1^\epsilon(x)\nabla u_1^\epsilon(t,x)) + {1 \over \epsilon}Q^\epsilon(x)(u_2^\epsilon(t,x)-u_1^\epsilon(t,x)) + q, \ \ x \in \Omega,\\
{\mathcal C}_{22}^\epsilon(x){\partial u_2^\epsilon(t,x) \over \partial t}=\text{div}(\kappa_2^\epsilon(x) \nabla u_2^\epsilon(t,x)) + {1 \over \epsilon}Q^\epsilon(x)(u_1^\epsilon(t,x)-u_2^\epsilon(t,x)) + q, \ \ x \in \Omega,
 \end{split}
 \end{equation} 
 with the Dirichlet boundary condition $u_1^\epsilon(t,x)=u_2^\epsilon(t,x)=0$ for $x\in\partial \Omega$, and with the initial condition $u_1^\epsilon(0,x)=g_1(x)$, $u_2^\epsilon(0,x)=g_2(x)$ where $g_1$ and $g_2$ are in $H$.
We assume there exist positive constants $\underline{C},\ \underline{\kappa}$ such that
\beq
{\mathcal C}_{ii}(x,y)\ge \underline{C},\ \kappa_i(x,y)\ge  \underline{\kappa}.
\label{eq:coercivity}
\eeq
%\jrp{ 
%  Multiplying $\phi_1$ and $\phi_2$ $\in {\mathcal C}_0^\infty(\Omega)$ to the first and second equations in (\ref{eq:main1}) respectively and integrating over $\Omega$, one has
In the weak form, equations \eqref{eq:main1} are of the form
  \begin{equation}
\label{eq:main65}
 \begin{split}
\int_{\Omega}{\mathcal C}_{11}^\epsilon{\partial u_1^\epsilon  \over \partial t}\phi_1  dx+ \int_{\Omega} \kappa_1^\epsilon\nabla u_1^\epsilon \cdot \nabla \phi_1  dx
-{1 \over \epsilon} \int_{\Omega} Q^\epsilon(u_2^\epsilon-u_1^\epsilon)\phi_1 dx = \int_{\Omega}q \phi_1  dx,\\
\int_{\Omega}{\mathcal C}_{22}^\epsilon{\partial u_2^\epsilon  \over \partial t} \phi_2  dx+ \int_{\Omega} \kappa_2^\epsilon\nabla u_2^\epsilon \cdot \nabla \phi_2  dx
-{1 \over \epsilon} \int_{\Omega} Q^\epsilon(u_1^\epsilon-u_2^\epsilon)\phi_2  dx = \int_{\Omega}q \phi_2  dx.
 \end{split}
 \end{equation}
 for all $\phi_1$ and $\phi_2$ in $C^\infty_0(\Omega)$.
We will prove in the appendix that system \eqref{eq:main1} has a unique solution $(u_1^\ep, u_2^\ep) \in L^2(0,T;V) \cap 
H^1(0,T; V')\times L^2(0,T;V) \cap H^1(0,T; V')$ which satisfies 
\beq
\label{eq:unibound}
||u_1^\ep||_{L^2(0,T;V)\cap H^1(0,T; V')}+ ||u_2^\ep||_{ L^2(0,T; V)\cap H^1(0,T; V')}
\leq C
\eeq
for a constant $C>0$ independent of $\ep$.
\subsection{Homogenization of multi-continuum system}
We study homogenization of this multi-continuum system by using the standard two scale asymptotic expansion.
We consider the two scale asymptotic expansion of $u_1^\epsilon$ and $u_2^\epsilon$
\begin{equation}
\label{eq:main3}
 \begin{split}
&u^{\epsilon}_1(t,x) = u_{10} (t,x,{x\over\epsilon})+ \epsilon u_{11}(t,x,{x\over\epsilon}) + \cdots\\
&u^{\epsilon}_2(t,x) = u_{20} (t,x,\xoe)+ \epsilon u_{21}(t,x,\xoe)+ \cdots,
 \end{split}
 \end{equation}
 where the functions $u_{1j}(t,x,y)$ and $u_{2j}(t,x,y)$ are $Y$-periodic with respect to $y $. 
 %$y = \frac{x}{\epsilon}$ is the fast variable and derivatives behave as 
% We note that $\nabla \rightarrow \nabla_x + \frac{1}{\epsilon} \nabla_y$. 
  %Fixing $x$, one inserts (\ref{eq:main3}) into (\ref{eq:main1}), and obtains
 From \eqref{eq:main1}, we have
 \begin{equation}
  \label{eq:main4}
 \begin{split}
{\mathcal C}_{11}&{\partial (u_{10} + \epsilon u_{11} + \cdots) \over \partial t}\\ &= (\div_x + {1 \over \epsilon} \div_y)(\kappa_1 (\nabla_x + {1 \over \epsilon} \nabla_y)(u_{10} + \epsilon u_{11}+ \cdots))
+ {1 \over \epsilon} Q(u_{20}+ \epsilon u_{21} - u_{10}- \epsilon u_{11}+ \cdots) + q, \\
{\mathcal C}_{22}&{\partial (u_{20} + \epsilon u_{21} + \cdots) \over \partial t}\\ &= (\div_x + {1 \over \epsilon} \div_y)(\kappa_2 (\nabla_x + {1 \over \epsilon} \nabla_y)(u_{20} + \epsilon u_{21}+ \cdots))
+ {1 \over \epsilon} Q(u_{10}+ \epsilon u_{11} - u_{20}- \epsilon u_{21}+ \cdots) + q.
 \end{split}
 \end{equation}
 Collecting the $\epsilon^{-2}$ terms, we obtain 
 \begin{equation}
  \label{eq:main5}
 \begin{split}
&\div_y (\kappa_1(x,y) \nabla_y u_{10}(t,x,y)) = 0\\
&\div_y (\kappa_2(x,y) \nabla_y u_{20}(t,x,y)) = 0.
 \end{split}
 \end{equation}
  From this, we deduce $u_{10}$ and $u_{20}$ are independent of $y$. Collecting the $\epsilon^{-1}$ terms we obtain
\begin{equation}
  \label{eq:main6}
 \begin{split}
&\div_y (\kappa_1 \nabla u_{10}) + \div_y (\kappa_1\nabla_y u_{11}) + Q(u_{20} - u_{10}) = 0\\
&\div_y (\kappa_2 \nabla u_{20}) + \div_y (\kappa_2\nabla_y u_{21}) + Q(u_{10} - u_{20}) = 0.
 \end{split}
 \end{equation} 
Therefore,
\begin{equation}
\label{eq:main7}
\begin{split}
u_{11}(t,x,y) = \sum_{i=1}^d N^i_1(x,y) \frac{\partial u_{10}(t,x)}{\partial x_i} + M_1(x,y) (u_{20}(t,x) - u_{10}(t,x))\\
u_{21}(t,x,y) = \sum_{i=1}^d N^i_2(x,y) \frac{\partial u_{20}(t,x)}{\partial x_i} + M_2(x,y) (u_{10} (t,x)- u_{20}(t,x)),
\end{split}
\end{equation} 
where $N^i_1(x,y)$, $N^i_2(x,y)$ ($i=1,\ldots,d$), $M_1(x,y)$ and $M_2(x,y)$, as functions of $y$ are the solutions 
%that belong to $H^1_\#(Y)/\mathbb{R}$ 
of the following cell problems respectively.
\begin{equation}
\label{eq:cell}
\begin{split}
&\div_y(\kappa_1(x,y)(e^i + \nabla_y N^i_1(x,y))) = 0\\
&\div_y(\kappa_1(x,y)\nabla_y M_1(x,y)) + Q(x,y) = 0\\
&\div_y(\kappa_2(x,y)(e^i + \nabla_y N^i_2(x,y))) = 0\\
&\div_y(\kappa_2(x,y)\nabla_y M_2(x,y)) + Q(x,y) = 0
\end{split}
\end{equation}
with the periodic boundary condition,
where $e^i$ is the $i$th standard basis vector of $\mathbb{R}^d$. Problems (\ref{eq:cell} (a),(c)) have a unique solution in $H^1_\#(Y)/\IR$; problems (\ref{eq:cell} (b),(d)) have a unique  solution since $\int_Y Q(x,y)  dy = 0$.
%The fast variable $y=x/\epsilon$ is defined in a unit cube and
%the equations (\ref{eq:cell}) is solved in a unit cube $Y$ with
%periodic boundary conditions. 
Collecting the $\epsilon^{0}$ terms, we have,
\begin{equation}
\label{eq:main8}
\begin{split}
{\mathcal C}_{11}{\partial u_{10} \over \partial t} = &\div_x(\kappa_1 \nabla u_{10}) + \div_y(\kappa_1 \nabla_x u_{11}) + \div_x(\kappa_1 \nabla_y u_{11}) + \div_y(\kappa_1\nabla_y u_{12}) + Q(u_{21} - u_{11}) + q \\
{\mathcal C}_{22}{\partial u_{20} \over \partial t} = &\div_x(\kappa_2 \nabla u_{20}) + \div_y(\kappa_2 \nabla_x u_{21}) + \div_x(\kappa_2 \nabla_y u_{21}) + \div_y(\kappa_2\nabla_y u_{22}) + Q(u_{11} - u_{21}) + q .
\end{split}
\end{equation} 
Integrating with respect to $y$ over $Y$ and using (\ref{eq:main7}), we have 
\begin{equation}
\label{eq:main9}
\begin{split}
& \left(\int_Y {\mathcal C}_{11} dy\right){\partial u_{10} \over \partial t}
= \div (\kappa_1^*\nabla u_{10}) 
+\div \bigg(\big(\int_Y \kappa_1 \nabla_y M_1  dy\big)(u_{20}-u_{10})\bigg)
\\ &+ \bigg(\big(\int_Y QN^i_2  dy\big)\frac{\partial u_{20}}{\partial x_i} - \big(\int_Y QN^i_1  dy\big)\frac{\partial u_{10}}{\partial x_i}\bigg) - \left(\int_YQ(M_1+M_2)  dy\right)(u_{20}-u_{10}) + q\\
&\left(\int_Y{\mathcal C}_{22}  dy\right) {\partial u_{20} \over \partial t} 
= \div (\kappa_2^*\nabla u_{20}) 
+\div \bigg(\big(\int_Y \kappa_2 \nabla_y M_2  dy\big)(u_{10}-u_{20})\bigg)
\\& + \bigg(\big(\int_Y QN^i_1  dy\big)\frac{\partial u_{10}}{\partial x_i} - \big(\int_Y QN^i_2  dy\big)\frac{\partial u_{20}}{\partial x_i}\bigg) - \left(\int_YQ(M_1+M_2) dy\right)(u_{10}-u_{20})+ q ,
\end{split}
\end{equation} 
where
\begin{equation}
 \label{eq:main10}
 \begin{split}
\kappa^*_{1ij}(x) =  \int_Y  \kappa_1 (x,y) (\delta_{ij} + {\partial  N^j_1(x,y)\over \partial y_i})  dy\\
\kappa^*_{2ij}(x) =  \int_Y  \kappa_2 (x,y) (\delta_{ij} + {\partial  N^j_2(x,y)\over \partial y_i})  dy.
\end{split}
\end{equation}
We note that $\kappa^*_{1ij}(x)$ and $\kappa^*_{2ij}(x)$ are standard homogenized coefficients for elliptic problems \cite{papanicolau1978asymptotic}. They are symmetric and positive definite (\cite{papanicolau1978asymptotic}).
We will show  in Section \ref{sec:3} that the initial conditions for $u_{10}$, $u_{20}$ are
\beq
u_{10}(0,x)=g_1(x),\ u_{20}(0,x)=g_2(x).
\label{eq:initcond}
\eeq
In the appendix, we show that the homogenized problem (\ref{eq:main9}) with these initial conditions has a unique solution. 
%Through out this paper, we denote the spaces $H$ and $V$ as $H$ and $V$ respectively. And we define the space $W$ by $V \times V$.

\textit{Remark.} The case where the continuum interacting term is scaled as $1/\ep$ considered in this paper has the most interesting homogenization limit, in comparison to other scalings, e.g. the $1/\ep^2$ scale case considered in \cite{park2019hierarchical}. It can be shown that the continuum interacting coefficient $- \int_YQ(M_1+M_2) dy$ in (\ref{eq:main9}) is always negative while the interaction coefficient $\frac{1}{\ep}Q$ in the two-scale problem can be both positive and negative due to Assumption \eqref{eq:Qaverage}.  The homogenized equation (\ref{eq:main9}) has convection terms, which is different from the original equation (\ref{eq:main1}). 
%Comparing to other scaling of interaction terms (\cite{park2019hierarchical}), $\frac{1}{\ep}$ scale we study in this paper gives the most interesting homogenization limit. 

%In this paper, we provide a rigorous justification of our homogenization. We have the following homogenization convergence result and error estimate.
We have the following homogenization results.
\begin{theorem}
\label{weakconv}
Assume that the solution $N_1^i$ and $N_2^i$ ($i=1,\ldots,d$) of cell problem (\ref{eq:cell} (a),(c)) belong to $C^2(\bar\Omega,C^2(\bar Y))$ and the coefficients $\kappa_1$ and $\kappa_2$ belong to $C^1(\bar\Omega,C^1(\bar Y))$.
The sequence $(u_1^\ep,u_2^\ep)$ of the solutions to (\ref{eq:main1}) converges weakly to $(u_{10}, u_{20})$ in $L^2(0,T;V)\times L^2(0,T;V)$, where $(u_{10},u_{20})$ is the solution of the homogenized equations \eqref{eq:main9} with initial conditions \eqref{eq:initcond}.
\end{theorem}
Since $\int_Y Q(x,y) dy = 0$, there is a vector function $\mathcal{Q}(x, y)$ which is periodic with respect to $y$ such that $Q(x,y) = \div_y \mathcal{Q}(x,y)$ (see \cite{jikov2012homogenization}).
We have the following result on homogenization convergence rate.
\begin{theorem}
\label{hom_error}
Assume $\kappa_1, \kappa_2 \in C^1(\bar{\Omega};C( \bar{Y}))$, $u_{10}, \ u_{20} \in C([0,T];C^2(\bar\Omega))\cap C^1([0,T];C^1(\bar\Omega)) $, 
%$u_{11} \in L^2(0,T; H^1(\Omega;H^1(Y)))$, $\mathcal{Q} \in$,
$N^i_k, M_k \in C^1(\bar{\Omega}, C^1(\bar{Y}))$, ($i=1,\ldots,d$, $k=1,2$),
${\mathcal Q} \in C^2(\bar{\Omega};C^1(\bar{Y}))^2$. Then we have
\beq
\bsp
||\nabla u_{11}^\ep - \nabla u_{10} - \nabla_y u_{11}(\cdot,\cdot,\frac{\cdot}{\ep}) ||_{L^2(0,T;H)}
+||\nabla u_{21}^\ep - \nabla u_{20} - \nabla_y u_{21}(\cdot,\cdot,\frac{\cdot}{\ep})||_{L^2(0,T;H)}
\leq c \ep^{1\over 2}
\end{split}
\eeq
where the constant $c$ is independent of $\ep$.
\end{theorem}

We prove Theorems \ref{weakconv} and \ref{hom_error} in Sections 3 and 4 respectively.
%%%%%%%%%%%%%%%%%%%%%%%%%%%Uniqueness and existence of solutions
%%%%%%%%%%%%%%%%%%%%%%%%%%%%%%%%%%%%%.    Homogenization.    %%%%%%%%%%%%%%%%%%%%%%%%%%%%%%%%%%%%%%%%%
\section{Proof of homogenization convergence}\label{sec:3}
In this section, we prove Theorem \ref{weakconv} on homogenization convergence for the solution of the two scale multi-continuum system \eqref{eq:main1}.
 %\todo{I don't see where Gronwall's lemma is used here} 
%%%%%%%%%%%%%%%%%%%%%%%%%%%%%%%%%%
%%%%%%%%%%%%%%%%%%%%%%%%%%%%%%%%%%
 From (\ref{eq:unibound}), there exists a subsequence of $(u_1^\epsilon,u_2^\epsilon)$, which we still denote by $(u_1^\epsilon, u_2^\epsilon)$ , $u_{10}$ and $u_{20}$ such that
\begin{equation}
\label{eq:main74}
 \begin{split}
u_1^{\epsilon} \rightharpoonup u_{10}, \ u_2^{\epsilon} \rightharpoonup u_{20} 
\enspace \textrm{in} \enspace L^2(0,T;V).
 \end{split}
 \end{equation}
 %%%%%%%%%%%%
  %%%%%%%%%%%%%
% We observe that $L^2(0,T;H) = L^2([0,T]\times \Omega)$.
% We set
%    \begin{equation}
%\label{eq:main77}
% \begin{split}
%\xi_1^{\epsilon}= \kappa^{\epsilon}_1(x) \nabla u_1^{\epsilon}(x,t).
% \end{split}
 %\end{equation}
%We can assume that
%\begin{equation}
%\label{eq:main78}
%\begin{split}
%\xi_1^{\epsilon} \rightarrow \xi \enspace \textrm{in} \enspace L^2(0,T;H) \enspace \textrm{weakly}.
%\end{split}
%\end{equation}
%
%
%
%Note that since $\kappa^{\epsilon}_1$ is $Y$ - periodic, one can verify that
%\begin{equation}
%\label{eq:main79}
%\begin{split}
%\kappa_1^{\epsilon} \rightarrow \int_{Y} \kappa_1 \textrm{d}y \enspace \textrm{weakly} \enspace \textrm{in} \enspace L^{2}([0,T]\times\Omega))
%\end{split}
%\end{equation}
%
%
%
%%%%%
 %%%%%
 %%%%%%%%%%%%%%%%%%%%%%%%%%%%%%%%%%%%%%%%%%%%%%%%%%%%%%%%%%%%%%%%%%%%%%%%%%%%%%%%%%%%%%%%
 We show that $(u_{10},u_{20})$ satisfies the homogenized problem \eqref{eq:main9}.
 Recall ${N^i_1}$, ${N^i_2}$, $M_1$ and $M_2$ in $H^1_{\#}(Y)$ as functions of $y$ are the solutions of the cell problems (\ref{eq:cell}).
Fixing $i=1,\ldots,d$, we consider 
\beq
\omega_1^\ep(x) = {x_i} +\ep N^i_1(x,\frac{x}{\epsilon})\ \  \mbox{and}\ \  \omega_2^\ep(x) = {x_i} +\ep N^i_2(x,\frac{x}{\epsilon}).
\eeq
%We define $\omega_1^\epsilon$ and $\omega_2^\epsilon$ as 
% \begin{equation}
%  \label{eq:main81}
% \begin{split}
%\omega_1^\epsilon(x) = \epsilon \omega_1(x,\frac{x}{\epsilon}),\ \omega_2^\epsilon(x) = %\epsilon \omega_2(x,\frac{x}{\epsilon}) .
% \end{split}
%\end{equation}
Under regularity conditions for $\kappa_1$, $\kappa_2$, $N_1^i$ and $N_2^i$, we have 
%\begin{equation}
%\label{eq:main81'}
 %\begin{split}
%&{\mathcal C}_{11}{\partial \omega_1^\epsilon (x) \over \partial t}-\text{div}(\kappa_1^\epsilon(x)\nabla \omega_1^\epsilon (x,t)) - {1 \over \epsilon^2}Q^\epsilon(x)(\omega_2^\epsilon(x,t)-\omega_1^\epsilon(x,t)) \\
%+ &{\mathcal C}_{22}{\partial \omega_2^\epsilon (x) \over \partial t}-\text{div}(\kappa_2^\epsilon(x) \nabla \omega_2^\epsilon (x,t)) - {1 \over \epsilon^2}Q^\epsilon(x)(\omega_1^\epsilon (x,t)-\omega_2^\epsilon(x,t)) = 0,
 %\end{split}
 %\end{equation}
 %Indeed,
 \begin{equation}
\label{eq:main81'''}
 \begin{split}
&-\div(\kappa_1^\epsilon(x)\nabla \omega^\epsilon_1(x)) \\
&= -\frac{1}{\epsilon} 
\text{div}_y(\kappa_1(x,\frac{x}{\epsilon}) (e^i + \nabla_y N_1^i (x,\frac{x}{\epsilon})))
-\epsilon\div_x (\kappa_1(x,\frac{x}{\epsilon})\nabla_xN_1^i(x,\frac{x}{\epsilon}))\\
&- \div_x(\kappa_1(x,\frac{x}{\epsilon})(e^i + \nabla_y N_1^i (x,\frac{x}{\epsilon})))
-\div_y(\kappa_1(x,\frac{x}{\epsilon})\nabla_x N_1^i(x,\frac{x}{\epsilon}))\\
%=-&\epsilon\int_\Omega\div_x (\kappa_1(x,\frac{x}{\epsilon})\nabla_x(N_1^i(x,\frac{x}{\epsilon})))\psi_1(x) dx 
%- \int_\Omega \div_x(\kappa_1(x,\frac{x}{\epsilon})(e^i + \nabla_y N_1^i (x,\frac{x}{\epsilon})))\psi_1(x) dx
%\\-& \int_\Omega\div_y(\kappa_1(x,\frac{x}{\epsilon})\nabla_x N_1^i(x,\frac{x}{\epsilon}))\psi_1(x) dx
 \end{split}
 \end{equation}
 and
 \begin{equation}
\label{eq:main81'''*}
 \begin{split}
&-\div(\kappa_2^\epsilon(x)\nabla \omega^\epsilon_2(x)) \\
&= -\frac{1}{\epsilon} 
\text{div}_y(\kappa_2(x,\frac{x}{\epsilon}) (e^i + \nabla_y N_2^i (x,\frac{x}{\epsilon})))
-\epsilon\div_x (\kappa_2(x,\frac{x}{\epsilon})\nabla_xN_2^i(x,\frac{x}{\epsilon}))\\
&- \div_x(\kappa_2(x,\frac{x}{\epsilon})(e^i + \nabla_y N_2^i (x,\frac{x}{\epsilon})))
-\div_y(\kappa_2(x,\frac{x}{\epsilon})\nabla_x N_2^i(x,\frac{x}{\epsilon})).
%=-&\epsilon\int_\Omega\div_x (\kappa_2(x,\frac{x}{\epsilon})\nabla_x(N_2^i(x,\frac{x}{\epsilon})))\psi_2(x) dx
%- \int_\Omega \div_x(\kappa_2(x,\frac{x}{\epsilon}) (e^i + \nabla_y N_2^i (x,\frac{x}{\epsilon})))\psi_2(x) dx
%\\-& \int_\Omega\div_y(\kappa_2(x,\frac{x}{\epsilon})\nabla_x N_2^i(x,\frac{x}{\epsilon}))\psi_2(x) dx
 \end{split}
 \end{equation}
 %where the last equalities are from (\ref{eq:dualcell}). 
 %Let us denote the right hand side of (\ref{eq:main81'''}) as $g^\epsilon$.
%Then using (\ref{eq:main81'''}), we have
%\begin{equation}
%\label{eq:main82}
 %\begin{split}
 %a^\epsilon((u_1^\epsilon,u_2^\epsilon),(\phi\omega_1^\epsilon,\phi\omega_2^\epsilon))-a^\epsilon((\omega_1^\epsilon,\omega_2^\epsilon),(\phi u_1^\epsilon,\phi u_2^\epsilon)) = ((q^\epsilon,q^\epsilon),(\phi\omega_1^\epsilon,\phi\omega_2^\epsilon))
 %\end{split}
 %\end{equation}
 %where $\phi \in \mathcal{C}^\infty_0(\Omega \times [0,T])$.
 %The left hand side of (\ref{eq:main82}) equals
% Let $\phi^\epsilon(x) = \phi(\frac{x}{\epsilon})$.
%Letting $\phi_1(x) = \phi(x)\omega_1^\epsilon(x)$,  
% $\phi_2(x) = \phi(x)\omega_2^\epsilon(x)$, where 
Let $\phi \in \mathcal{C}^\infty_0(\Omega)$. From (\ref{eq:main65}), we have
\begin{equation}
\label{eq:main82*}
\int_{\Omega}{\mathcal C}_{11}^\epsilon{\partial u_1^\epsilon  \over \partial t}\phi\omega_1^\epsilon  dx
+ \int_{\Omega} \kappa_1^\epsilon\nabla u_1^\epsilon \cdot \nabla (\phi\omega_1^\epsilon)  dx
-\int_{\Omega} {1 \over \epsilon}Q^\epsilon(u_2^\epsilon-u_1^\epsilon)\phi\omega_1^\epsilon dx
=\int_{\Omega}q \phi\omega_1^\epsilon  dx,
\end{equation}
and
\begin{equation}
\label{eq:main82a}
\int_{\Omega}{\mathcal C}_{22}^\epsilon{\partial u_2^\epsilon  \over \partial t} \phi\omega_2^\epsilon  dx
+ \int_{\Omega} \kappa_2^\epsilon\nabla u_2^\epsilon \cdot \nabla (\phi\omega_2^\epsilon)  dx
-\int_{\Omega} {1 \over \epsilon}Q^\epsilon(u_1^\epsilon-u_2^\epsilon)\phi\omega_2^\epsilon  dx
= \int_{\Omega}q \phi\omega_2^\epsilon  dx.
 \end{equation}
 Multiplying (\ref{eq:main81'''}) and (\ref{eq:main81'''*}) by $\phi u_1^\epsilon$ and $\phi u_2^\epsilon$ respectively and integrate over $\Omega$ we have
 \begin{equation}
 \begin{split}
\label{eq:main82**}
&\int_{\Omega} \kappa_1^\epsilon\nabla \omega_1^\epsilon \cdot \nabla (\phi u_1^\epsilon)  dx
= -\epsilon \int_{\Omega} \div_x (\kappa_1(x,\xoe) \nabla_x N_1^i(x,\xoe)) \phi u_1^\epsilon dx\\
&- \int_{\Omega} \div_x(\kappa_1(x,\xoe)(e^i + \nabla_y N_1^i (x,\xoe)))\phi u_1^\epsilon  dx
-\int_{\Omega} \div_y( \kappa_1(x,\xoe)\nabla_x N_1^i(x,\xoe) )\phi u_1^\epsilon  dx,
 \end{split}
 \end{equation}
 and
  \begin{equation}
\label{eq:main82***}
 \begin{split}
  &\int_{\Omega} \kappa_2^\epsilon\nabla \omega_2^\epsilon \cdot \nabla(\phi u_2^\epsilon)  dx
= - \epsilon \int_{\Omega} \div_x( \kappa_2(x,\xoe)\nabla_x N_2^i(x,\xoe))\phi u_2^\epsilon  dx\\
&- \int_{\Omega} \div_x(\kappa_2(x,\xoe)(e^i + \nabla_y N_2^i (x,\xoe)))\phi u_2^\epsilon  dx
- \int_{\Omega} \div_y (\kappa_2(x,\xoe)\nabla_x N_2^i(x,\xoe)) \phi u_2^\epsilon  dx.
  \end{split}
 \end{equation}
 Let $\psi\in C^\infty_0(0,T)$. 
 Subtracting (\ref{eq:main82**}), (\ref{eq:main82***}) from (\ref{eq:main82*}) and \eqref{eq:main82a} respectively, we obtain
 \begin{equation}
\label{eq:main82'}
 \begin{split}
&\int_0^T\int_\Omega {\mathcal C}_{11}^\epsilon{\partial u_1^\epsilon\over \partial t} \phi\psi\omega_1^\epsilon  dx dt
+ \int_0^T\int_\Omega \kappa_1^\epsilon \nabla u_1^\epsilon \cdot\nabla \phi \omega_1^\epsilon \psi  dx dt
- \int_0^T\int_{\Omega} {1 \over \epsilon}Q^\epsilon(u_2^\epsilon-u_1^\epsilon)\phi\omega_1^\epsilon \psi  dx dt\\
&\qquad\qquad-  \int_0^T\int_\Omega \kappa_1^\epsilon \nabla \omega_1^\epsilon\cdot (\nabla\phi u_1^\epsilon)\psi  dx  dt\\
%= \int_0^T\int_\Omega {\mathcal C}_{11}^\epsilon{\partial u_1^\epsilon \over \partial t} \phi\psi\omega_1^\epsilon  dx dt
%+ \int_0^T\int_\Omega \kappa_1^\epsilon \nabla u_1^\epsilon(\nabla \phi) \omega_1^\epsilon \psi dx dt
%- \int_0^T\int_\Omega \kappa_1^\epsilon \nabla \omega_1^\epsilon (\nabla\phi) u_1^\epsilon\psi  dx dt\\
&= \int_0^T\int_\Omega q \phi \omega_1^\epsilon\psi dxdt
+\epsilon\int_0^T \int_{\Omega} \div_x \left(\kappa_1(x,{x\over\ep}) \nabla_x N_1^i(x,{x\over\ep})\right) \phi u_1^\epsilon\psi dx dt\\
&+ \int_0^T\int_{\Omega} \div_x\left(\kappa_1(x,{x\over\ep})(e^i + \nabla_y N_1^i(x,{x\over\ep}))\right)\phi u_1^\epsilon\psi  dx dt
+\int_0^T\int_{\Omega} \div_y\left( \kappa_1(x,{x\over\ep})\nabla_x N_1^i(x,{x\over\ep}) \right)\phi u_1^\epsilon \psi dx dt
\end{split}
\end{equation}
and
 \begin{equation}
\label{eq:main82'*}
 \begin{split}
&\int_0^T\int_\Omega {\mathcal C}_{22}^\epsilon{\partial u_2^\epsilon\over \partial t} \phi\psi\omega_2^\epsilon  dx dt
+ \int_0^T\int_\Omega \kappa_2^\epsilon \nabla u_2^\epsilon\cdot \nabla \phi \omega_2^\epsilon \psi  dx dt
- \int_0^T\int_{\Omega} {1 \over \epsilon}Q^\epsilon(u_1^\epsilon-u_2^\epsilon)\phi\omega_2^\epsilon \psi  dx dt\\
&\qquad\qquad-  \int_0^T\int_\Omega \kappa_2^\epsilon \nabla \omega_2^\epsilon\cdot (\nabla\phi u_2^\epsilon)\psi  dx  dt\\
%= \int_0^T\int_\Omega {\mathcal C}_{11}^\epsilon{\partial u_1^\epsilon \over \partial t} \phi\psi\omega_1^\epsilon  dx dt
%+ \int_0^T\int_\Omega \kappa_1^\epsilon \nabla u_1^\epsilon(\nabla \phi) \omega_1^\epsilon \psi dx dt
%- \int_0^T\int_\Omega \kappa_1^\epsilon \nabla \omega_1^\epsilon (\nabla\phi) u_1^\epsilon\psi  dx dt\\
&= \int_0^T\int_\Omega q \phi \omega_2^\epsilon\psi  dx dt
+\epsilon\int_0^T \int_{\Omega} \div_x \left(\kappa_2(x,{x\over\ep}) \nabla_x N_2^i(x,{x\over\ep})\right) \phi u_2^\epsilon\psi dx dt\\
&+ \int_0^T\int_{\Omega} \div_x\left(\kappa_2(x,{x\over\ep})(e^i + \nabla_y N_2^i(x,{x\over\ep}) )\right)\phi u_2^\epsilon\psi  dx dt
+\int_0^T\int_{\Omega} \div_y\left( \kappa_2(x,{x\over\ep})\nabla_x N_2^i(x,{x\over\ep}) \right)\phi u_2^\epsilon \psi dx dt.
\end{split}
\end{equation}
We have the following lemma.
\begin{lemma}
\label{lemma4}
The functions $\int_0^T\psi(t)u_1^\ep(t,x)dt$ and  $\int_0^T\psi(t)u_2^\ep(t,x)dt$ converge strongly in $H$ to \\ $\int_0^T\psi(t) u_{10}(t,x)dt$
and $\int_0^T\psi(t) u_{20}(t,x)dt$ respectively, for $\psi \in C^\infty_0 (0,T) $.
\end{lemma}
{\it Proof\ \ } This is the standard result in Jikov et al. \cite{jikov2012homogenization}. As $u_1^\ep$ is uniformly bounded in $L^2(0,T;V)$, we have that $\int_0^T\psi(t)u_1^\ep(t,x) dt$ is uniformly bounded in $V$. Thus we can extract a subsequence which converges weakly in $V$ and strongly in $H$. As for all $\phi\in C^\infty_0(\Omega)$, 
\[
\int_\Omega\int_0^T\psi(t)u_1^\ep(t,x)\phi(x) dt dx\to \int_\Omega\int_0^T\psi(t)u_{10}(t,x)\phi(x) dt dx,
\]
the limit is $\int_\Omega\psi(t)u_{10}(t,x) dt$.\hfill$\Box$\\
We have
\[
\begin{split}
\int_0^T\int_\Omega C_{11}^\ep{\partial u_1^\ep\over\partial t}\phi\psi\omega_1^\ep  dx dt=-\int_\Omega C_{11}^\ep\left(\int_0^T u_1^\ep{\partial\psi\over\partial t}dt\right)\phi\omega_1^\ep  dx.
\end{split}
\]
{Note that} $C_{11}^\ep$ converges weakly to $\int_YC_{11}(x,y) dy$ in $H$, and $\int_0^Tu_1^\ep{\partial\psi\over\partial t} dt$ converges strongly to $\int_0^Tu_{10}{\partial\psi\over\partial t} dt$ in $H$. Thus 
\[
\begin{split}
\lim_{\ep\to 0}\int_0^T\int_\Omega C_{11}^\ep{\partial u_1^\ep\over\partial t}\phi\psi\omega_1^\ep  dx dt=
-\int_0^T\int_\Omega\left(\int_YC_{11}(x,y)dy\right)u_{10}{\partial\psi\over\partial t}\phi x_i  dx dt \\
=\int_0^T\int_\Omega\left(\int_YC_{11}(x,y)dy\right){\partial u_{10}\over\partial t}\psi\phi x_i  dx dt.
\end{split}
\]
%\todo{please edit the rest following what I have done above. We need to multiply everything by the function $\psi(t)$ and take the integral over $(0,T)$. As you have seen, doing this we can get a strong convergence in $L^2(D)$, which I denote as $H$ (please add this definition somewhere).} 
Note that we have 
\begin{equation}
\label{eq:main82''}
 \begin{split}
 \kappa_1^\epsilon(x) \nabla \omega_1^\epsilon(x) 
 = \kappa_1(x,\frac{x}{\epsilon})\big((e^i + \nabla_y N_1^i (x,\frac{x}{\epsilon}))+ \epsilon \nabla_x N^i_1(x,\frac{x}{\epsilon})\big),\\
 \kappa_2^\epsilon(x) \nabla \omega_2^\epsilon(x)
 = \kappa_2(x,\frac{x}{\epsilon})\big((e^i + \nabla_yN_2^i (x,\frac{x}{\epsilon}))+ \epsilon \nabla_x N^i_2(x,\frac{x}{\epsilon})\big).
  \end{split}
 \end{equation}
 Also, note that due to $Y$-periodicity of $\kappa$ and $N^i$, we have
\begin{equation}
\label{eq:main82'''}
 \begin{split}
&\kappa_1(x,\frac{x}{\epsilon}) (e^i + \nabla_y N_1^i (x,\frac{x}{\epsilon}))
 \rightharpoonup \int_Y \kappa_1(x, y)(e^i + \nabla_y N_1^i (x,y))  dy,\\
&\kappa_2(x,\frac{x}{\epsilon}) (e^i + \nabla_y N_2^i (x,\frac{x}{\epsilon}))
 \rightharpoonup \int_Y \kappa_2(x, y)(e^i + \nabla_y N_2^i (x,y))  dy
\enspace \textrm{in} \enspace H.
  \end{split}
 \end{equation}
%Since $\kappa_i$, $N_1^i$, $N_2^i$ and $\phi$ are independent of $t$, 
Passing to the limit in (\ref{eq:main82'}), (\ref{eq:main82'*}),
we obtain from Lemma \ref{lemma4},
\begin{equation}
\label{eq:main82''''}
 \begin{split}
 \int_0^T \int_\Omega \left(\int_Y {\mathcal C}_{11}  dy\right){\partial u_{10} \over \partial t} \phi \psi x_i  dx  dt
- \int_0^T\int_\Omega \left(\int_Y \kappa_1(e^i + \nabla_y N_1^i)  dy\right) \cdot\nabla \phi \psi u_{10}  dx dt\\
+ \displaystyle \lim_{\epsilon \to 0 }\left(\int_0^T\int_\Omega \kappa_1^\epsilon \nabla u_1^\epsilon \cdot\nabla \phi\psi \omega_1^\ep  dx  dt
-{1 \over \epsilon}\int_0^T\int_{\Omega} Q^\epsilon(u_2^\epsilon-u_1^\epsilon)\phi \omega_1^\ep \psi  dx dt\right)\\
=\int_0^T\int_\Omega q \phi x_i \psi  dx dt
+ \int_0^T\int_{\Omega} \div\left(\int_Y \kappa_1(e^i + \nabla_y N_1^i ) dy\right)\phi u_{10}\psi  dx dt\\
 \end{split}
 \end{equation}
 and
 \begin{equation}
\label{eq:main83}
 \begin{split}
  \int_0^T \int_\Omega\left( \int_Y {\mathcal C}_{22}  dy\right){\partial u_{20} \over \partial t} \phi \psi x_i  dx  dt
- \int_0^T\int_\Omega\left( \int_Y \kappa_2(e^i + \nabla_y N_2^i)  dy\right) \cdot\nabla \phi \psi u_{20}  dx dt\\
+ \displaystyle \lim_{\epsilon \to 0 } 
\left(\int_0^T\int_\Omega \kappa_2^\epsilon \nabla u_2^\epsilon \cdot\nabla \phi\psi \omega_2^\ep  dx  dt
-{1\over\epsilon}\int_0^T\int_{\Omega} Q^\epsilon(u_1^\epsilon-u_2^\epsilon)\phi \omega_2^\ep\psi  dx dt\right)\\
=\int_0^T\int_\Omega q \phi x_i \psi  dx dt
+ \int_0^T\int_{\Omega} \div\left(\int_Y \kappa_2(e^i + \nabla_y N_2^i ) dy\right)\phi u_{20}\psi  dx dt.\\
 \end{split}
 \end{equation}
% Note that  $\int_{\Omega} q \phi \omega_k^\epsilon  dx \rightarrow \int_\Omega  q \phi x_i  dx$ since $\omega_k^\epsilon\phi \rightarrow x_i\phi$ in $H$ weakly.
% Here, we used the fact that
%%%%%%%%%%%%%%%%%%%%%%%%%%%%%%%%%%%%%%%%%%%%%% 
%\begin{equation}
%\label{eq:main84'}
% \begin{split}
 %\int_\Omega {\mathcal C}_{11}{\partial u_0 \over \partial t} \phi x_i
%+  \displaystyle \lim_{\epsilon \to 0 }\int_\Omega \kappa_1^\epsilon \nabla u_1^\epsilon \cdot(\nabla \phi) x_i
%+ \int_\Omega {\mathcal C}_{22}{\partial u_0 \over \partial t} \phi x_i
%+  \displaystyle \lim_{\epsilon \to 0 }\int_\Omega \kappa_2^\epsilon \nabla u_2^\epsilon \cdot(\nabla \phi) x_i\\
%=2 \int_\Omega q\phi x_i
 %- \displaystyle \lim_{\epsilon \to 0} \int_{\Omega} \kappa_1(x,y)\nabla_y (y_i+N_1^i) \cdot (\nabla_x u_1)\phi
 %- \displaystyle \lim_{\epsilon \to 0} \int_{\Omega} \kappa_2(x,y)\nabla_y (y_i+N_2^i) \cdot (\nabla_x u_2)\phi
 %\end{split}
 %\end{equation}
Letting $\phi_1$ and $\phi_2$ in (\ref{eq:main65}) be $\phi x_i$ for $\phi\in C^\infty_0(\Omega)$, we get
\begin{equation}
\label{eq:main68''}
 \begin{split}
 \int_0^T\int_{\Omega}{\mathcal C}_{11}^\epsilon{\partial u_1^\epsilon  \over \partial t} \phi \psi x_i  dx  dt
+  \int_0^T\int_{\Omega} \kappa_1^\epsilon\nabla u_1^\epsilon \cdot \nabla (\phi x_i)\psi  dx dt
- {1 \over \epsilon}\int_0^T\int_{\Omega} Q^\epsilon(u_2^\epsilon-u_1^\epsilon) \phi \psi x_i  dx dt\\
%+ \int_0^T\int_{\Omega}{\mathcal C}_{22}^\epsilon{\partial u_2^\epsilon  \over \partial t} \phi \psi x_i  dx dt 
% +  \int_0^T\int_{\Omega} \kappa_2^\epsilon\nabla u_2^\epsilon \cdot \nabla (\phi x_i)\psi  dx dt
%- \int_0^T\int_{\Omega} {1 \over \epsilon^2}Q^\epsilon(u_1^\epsilon-u_2^\epsilon) \phi \psi x_i  dx dt\\
=  \int_0^T\int_{\Omega}q \phi\psi  x_i  dx dt.
 \end{split}
 \end{equation} 
 Passing to the limit when $\ep\to 0$, we obtain
%\begin{equation}
%\label{eq:main84''}
% \begin{split}
%\int_\Omega {\mathcal C}_{11}{\partial u_1^\epsilon \over \partial t} \phi x_i
%+ \int_\Omega \kappa_1^\epsilon \nabla u_1^\epsilon(\nabla \phi) x_i
%- \int_\Omega \kappa_1^\epsilon e^i (\nabla\phi) u_1^\epsilon\\
%+ \int_\Omega {\mathcal C}_{22}{\partial u_2^\epsilon \over \partial t} \phi x_i
%+ \int_\Omega \kappa_2^\epsilon \nabla u_2^\epsilon (\nabla\phi) x_i
%- \int_\Omega \kappa_2^\epsilon e^i (\nabla\phi) u_2^\epsilon\\
%= \int_\Omega q (\phi x_i) + \int_\Omega q (\phi x_i) 
%- \int_{\Omega} \kappa_1(x,y) e^i \nabla_x(\phi u_1^\epsilon)
%- \int_{\Omega} \kappa_2(x,y) e^i \nabla_x(\phi u_2^\epsilon)
%\end{split}
%\end{equation}
\begin{equation}
\label{eq:main84'''}
 \begin{split}
\int_0^T\int_\Omega\left( \int_Y {\mathcal C}_{11}  dy\right){\partial u_{10} \over \partial t} \phi\psi x_i  dx dt
+\displaystyle \lim_{\epsilon \to 0} 
\left(\int_0^T\int_\Omega \kappa_1^\epsilon \nabla u_1^\epsilon\cdot\nabla (\phi x_i)\psi   dx dt
-{1 \over \epsilon}\int_0^T\int_{\Omega} Q^\epsilon(u_2^\epsilon-u_1^\epsilon) \phi \psi x_i  dx dt\right)\\
%+ \int_0^T\int_\Omega \int_Y {\mathcal C}_{22}  dy{\partial u_0 \over \partial t} \phi\psi  x_i  dx dt\\
%+ \displaystyle \lim_{\epsilon \to 0} 
%\int_0^T\int_\Omega \kappa_2^\epsilon \nabla u_2^\epsilon \nabla (\phi x_i) \psi  dx dt
=\int_0^T \int_\Omega q \phi x_i\psi   dx dt.
 \end{split}
 \end{equation}
 Subtracting (\ref{eq:main84'''}) from  (\ref{eq:main82''''}), one obtains
 %\beq
 %\bsp
 %-\int_0^T\int_\Omega \int_Y \kappa_1(e^i + \nabla_y N_1^i)  dy \cdot(\nabla \phi) \psi u_{10}  dx dt\\
%+ \displaystyle \lim_{\epsilon \to 0 }\big(\int_0^T\int_\Omega \kappa_1^\epsilon \nabla u_1^\epsilon \cdot(\nabla \phi)\psi \omega_1^\ep  dx  dt
%-{1 \over \epsilon}\int_0^T\int_{\Omega} Q^\epsilon(u_2^\epsilon-u_1^\epsilon)\phi \omega_1^\ep \psi  dx dt\big)\\
%-\displaystyle \lim_{\epsilon \to 0} 
%\big(\int_0^T\int_\Omega \kappa_1^\epsilon \nabla u_1^\epsilon\cdot\nabla (\phi x_i)\psi   dx dt
%-{1 \over \epsilon}\int_0^T\int_{\Omega} Q^\epsilon(u_2^\epsilon-u_1^\epsilon) \phi \psi x_i  dx dt\big)\\
%= \int_0^T\int_{\Omega} \div(\int_Y \kappa_1(e^i + \nabla_y N_1^i ) dy)\phi u_{10}\psi  dx dt.\\
%\end{split}
%\eeq
%This implies
 \beq
 \bsp
&- \int_0^T\int_\Omega\left( \int_Y \kappa_1(e^i + \nabla_y N_1^i)  dy\right) \cdot\nabla \phi \psi u_{10}  dx dt\\
&- \displaystyle \lim_{\epsilon \to 0 }\left(\int_0^T\int_\Omega \kappa_1^\epsilon \nabla u_1^\epsilon \cdot e^i \phi \psi   dx  dt
+\int_0^T\int_{\Omega} Q^\epsilon(u_2^\epsilon-u_1^\epsilon) N_1^i(x,{x\over\ep})\phi \psi  dx dt\right)\\
&=- \int_0^T\int_{\Omega}\left( \int_Y \kappa_1(e^i + \nabla_y N_1^i ) dy\right)\cdot \nabla(u_{10}\phi)\psi  dx dt.\\
\end{split}
\eeq
Using Lemma \ref{lemma4}, we get
 \beq
 \bsp
& -\int_0^T\int_\Omega\left( \int_Y \kappa_1(e^i + \nabla_y N_1^i)  dy\right) \cdot\nabla \phi \psi u_{10}  dx dt\\
&- \displaystyle \lim_{\epsilon \to 0 }\int_0^T\int_\Omega \kappa_1^\epsilon \nabla u_1^\epsilon \cdot e^i \phi \psi   dx  dt
-\int_0^T\int_{\Omega}\left( \int_Y Q N_1^i  dy\right) (u_{20}-u_{10}) \phi \psi  dx dt\\
&=-\int_0^T\int_{\Omega} \int_Y \kappa_1(e^i + \nabla_y N_1^i ) dy \cdot\nabla (u_{10}\phi)\psi  dx dt.\\
\end{split}
\eeq
 From this, we have
 \begin{equation}
\label{eq:main84''''}
 \begin{split}
  &\displaystyle \lim_{\epsilon \to 0 }
\int_0^T\int_\Omega \kappa_1^\epsilon \nabla u_1^\epsilon \cdot e^i \phi\psi dx  dt=
-\int_0^T\int_\Omega \left(\int_Y \kappa_1(e^i + \nabla_y N_1^i)  dy\right) \cdot\nabla \phi \psi u_{10}  dx dt\\
&+\int_0^T\int_{\Omega}\left( \int_Y \kappa_1(e^i + \nabla_y N_1^i ) dy\right) \cdot \nabla(u_{10}\phi)\psi  dx dt
- \int_0^T\int_{\Omega}\left( \int_Y Q N_1^i  dy \right)(u_{20}-u_{10}) \phi \psi  dx  dt\\
&= \int^T_0\int_\Omega \left(\int_Y \kappa_1(e^i + \nabla_y N_1^i) dy\right)\cdot\nabla u_{10} \phi \psi dx dt
+ \int_0^T\int_{\Omega}\left( \int_Y \kappa_1 \nabla_y M_1 \cdot e^idy\right) (u_{20}-u_{10}) \phi\psi  dx  dt,\\
\end{split}
 \end{equation}
where we use (\ref{eq:cell} (a),(b)) for the last term of \eqref{eq:main84''''}.
 %Since $\kappa_1$, $\kappa_2$, $N_1^i$ and $N_2^i$ are independent of $t$, by Lemma \ref{lemma2}, we have
 Similarly,
  \begin{equation}
\label{eq:main84*}
 \begin{split}
&\displaystyle \lim_{\epsilon \to 0 }
\int_0^T\int_\Omega \kappa_2^\epsilon \nabla u_2^\epsilon \cdot e^i \phi\psi dx  dt\\
&= \int^T_0\int_\Omega \left(\int_Y \kappa_2(e^i + \nabla_y N_2^i) dy\right)\cdot\nabla u_{20} \phi  dx\psi dt
+ \int_0^T\int_{\Omega} \left(\int_Y \kappa_2 \nabla_y M_2 \cdot e^i  dy\right) (u_{10}-u_{20}) \phi \psi  dx  dt.\\
 \end{split}
 \end{equation}
 From \eqref{eq:main84''''}, one obtains
 \begin{equation}
\label{eq:main95-1}
 \begin{split}
&  \displaystyle \lim_{\epsilon \to 0}
  \int_0^T\int_\Omega \kappa_1^\epsilon \nabla u_1^\epsilon \cdot \nabla \phi\psi   dx   dt\\
&  =  \int^T_0\int_\Omega \left(\int_Y \kappa_1(e^i + \nabla_y N_1^i) dy\right)\cdot\nabla u_{10} \frac{\partial \phi}{\partial x_i}\psi dx dt
+ \int_0^T\int_{\Omega}\left( \int_Y \kappa_1 \nabla_y M_1 \cdot e^i  dy\right) (u_{20}-u_{10})  \frac{\partial \phi}{\partial x_i}\psi  dx dt\\
% \end{split}
%\end{equation}
%Then it can be written as
%&  \displaystyle \lim_{\epsilon \to 0}
%  \int_0^T\int_\Omega \kappa_1^\epsilon(x) \nabla u_1^\epsilon(x) \cdot \nabla \phi\psi   dx   dt\\
%&  =  \int^T_0\int_\Omega \kappa_{1ij}^{*T} \frac{\partial u_{10}}{\partial x_j} \frac{\partial \phi}{\partial x_i} dx\psi dt
%+ \int_0^T\int_{\Omega} \int_Y \kappa_1 \nabla_y M_1   dy \cdot \nabla \phi\psi (u_{20}-u_{10})  dx   dt,
&=\int_0^T\int_\Omega  \kappa_1^{*} \nabla u_{10} \cdot \nabla \phi\psi   dx dt+ \int_0^T\int_{\Omega}\left( \int_Y \kappa_1 \nabla_y M_1   dy\right) \cdot \nabla \phi\psi (u_{20}-u_{10})  dx   dt,
 \end{split}
 \end{equation}
where we have used the standard result on the symmetry of the homogenized coefficient $\kappa_1^*$ defined in \eqref{eq:main10} (see, e.g., \cite{papanicolau1978asymptotic}).
 %$\kappa_{1ij}^{*T} =  \int_Y \kappa_1(\delta_{ij} + \frac{\partial N_1^i}{\partial y_j}) dy$.
% \begin{equation}
%\label{eq:main95}
% \begin{split}
%&  \displaystyle \lim_{\epsilon \to 0}
%  \int_0^T\int_\Omega \kappa_1^\epsilon(x) \nabla u_1^\epsilon(x) \cdot \nabla \phi\psi   dx   dt\\
%&  = \int_0^T\int_\Omega \big( \kappa_1^{*T} \nabla u_{10}(x) \big)\cdot \nabla \phi\psi   dx dt
% + \int_0^T\int_{\Omega} \int_Y \kappa_1 \nabla_y M_1   dy \cdot \nabla \phi\psi (u_{20}-u_{10})  dx   dt,
% \end{split}
% \end{equation}
%where $ \kappa_1^{*T}$ is a matrix such that $ \big(\kappa_1^{*T} \big)_{ij}=\kappa_{1ij}^{*T}$. 
Similarly, we deduce
 \begin{equation}
\label{eq:main96}
 \begin{split}
&  \displaystyle \lim_{\epsilon \to 0}
  \int_0^T\int_\Omega \kappa_2^\epsilon \nabla u_2^\epsilon \cdot \nabla \phi\psi   dx   dt\\
&  = \int_0^T\int_\Omega  \kappa_2^{*} \nabla u_{20}\cdot \nabla \phi\psi   dx dt
 + \int_0^T\int_{\Omega}\left( \int_Y \kappa_2 \nabla_y M_2   dy\right) \cdot \nabla \phi\psi (u_{10}-u_{20})  dx   dt,
 \end{split}
 \end{equation}
where  $\kappa_2^*$ is defined in \eqref{eq:main10}. 
We define $\gamma_1^\epsilon$ and $\gamma_2^\epsilon$ as 
 \begin{equation}
  \label{eq:main100}
 \begin{split}
\gamma_1^\epsilon(x) = \epsilon M_1(x,\frac{x}{\epsilon}),\
 \gamma_2^\epsilon(x) = \epsilon M_2(x,\frac{x}{\epsilon}) .
 \end{split}
\end{equation}
Under the smoothness conditions for $\kappa_1$, $M_1$, we have 
 \begin{equation}
\label{eq:main101}
 \begin{split}
-\text{div}(\kappa_1^\epsilon(x)\nabla \gamma^\epsilon_1(x)) 
= -\frac{1}{\epsilon}
\text{div}_y(\kappa_1(x,\frac{x}{\epsilon}) \nabla_y M_1 (x,\frac{x}{\epsilon})) 
-\epsilon \div_x (\kappa_1(x,\frac{x}{\epsilon})\nabla_xM_1(x,\frac{x}{\epsilon})) \\
-\div_x(\kappa_1(x,\frac{x}{\epsilon})\nabla_y M_1 (x,\frac{x}{\epsilon}))
-\div_y(\kappa_1(x,\frac{x}{\epsilon})\nabla_x M_1(x,\frac{x}{\epsilon})).\\
%=-&\epsilon\int_\Omega\div_x (\kappa_1(x,\frac{x}{\epsilon})\nabla_x(N_1^i(x,\frac{x}{\epsilon})))\psi_1(x) dx 
%- \int_\Omega \div_x(\kappa_1(x,\frac{x}{\epsilon})(e^i + \nabla_y N_1^i (x,\frac{x}{\epsilon})))\psi_1(x) dx
%\\-& \int_\Omega\div_y(\kappa_1(x,\frac{x}{\epsilon})\nabla_x N_1^i(x,\frac{x}{\epsilon}))\psi_1(x) dx
 \end{split}
 \end{equation}
 Letting $\phi_1(x) = \phi(x)\gamma_1^\epsilon(x)$
 where $\phi \in \mathcal{C}^\infty_0(\Omega)$ in (\ref{eq:main65}), we obtain
\begin{equation}
\label{eq:main102}
 \begin{split}
\int_{\Omega}{\mathcal C}_{11}^\epsilon{\partial u_1^\epsilon  \over \partial t}\phi\gamma_1^\epsilon  dx
+ \int_{\Omega} \kappa_1^\epsilon\nabla u_1^\epsilon \cdot \nabla (\phi\gamma_1^\epsilon)  dx
-\int_{\Omega} {1 \over \epsilon}Q^\epsilon(u_2^\epsilon-u_1^\epsilon)\phi\gamma_1^\epsilon dx
=\int_{\Omega}q \phi\gamma_1^\epsilon  dx.\\
%\int_{\Omega}{\mathcal C}_{22}^\epsilon{\partial u_2^\epsilon  \over \partial t} \phi\omega_2^\epsilon  dx
%+ \int_{\Omega} \kappa_2^\epsilon\nabla u_2^\epsilon \cdot \nabla (\phi\omega_2^\epsilon)  dx
%-\int_{\Omega} {1 \over \epsilon}Q^\epsilon(u_1^\epsilon-u_2^\epsilon)\phi\omega_2^\epsilon  dx
%= \int_{\Omega}q \phi\omega_2^\epsilon  dx.
 \end{split}
 \end{equation}
 Let $\psi\in C_0^\infty((0,T))$. From (\ref{eq:main101}) we have
 %by $\phi u_1^\ep$ and integrate over $\Omega$ and $[0,T]$ 
\beq
\label{eq:main102'}
\bsp
\int^T_0\int_{\Omega}\div(\kappa_1^\ep \nabla \gamma_1^\ep) \phi(x) u_1^\ep\psi(t)   dx dt
= -\frac{1}{\epsilon}\int^T_0\int_\Omega Q(x,\xoe)\phi(x) u_1^\ep\psi(t) dx dt\\
%+\int^T_0\int_\Omega Q(x,\frac{x}{\epsilon})(M_2(x,\frac{x}{\epsilon})-M_1(x,\frac{x}{\epsilon})\phi u_1^\ep\psi dx dt\\
+\epsilon\int^T_0\int_\Omega\div_x \big(\kappa_1(x,\frac{x}{\epsilon})\nabla_xM_1(x,\frac{x}{\epsilon})\big)\phi(x) u_1^\ep\psi(t) dx dt
+\int^T_0\int_\Omega \div_x\big(\kappa_1(x,\frac{x}{\epsilon})\nabla_y M_1 (x,\frac{x}{\epsilon})\big)\phi(x) u_1^\ep\psi(t) dx  dt
\\+\int^T_0\int_\Omega\div_y\big(\kappa_1(x,\frac{x}{\epsilon})\nabla_x M_1(x,\frac{x}{\epsilon})\big)\phi(x) u_1^\ep\psi(t) dx dt.\\
\end{split}
\eeq
% where $\psi(t) \in {\mathcal C}_0^\infty(0,T)$. 
Adding \eqref{eq:main102} and (\ref{eq:main102'}), we have
 \begin{equation}
\label{eq:main103}
 \begin{split}
\int^T_0\int_{\Omega}{\mathcal C}_{11}^\epsilon{\partial u_1^\epsilon  \over \partial t}\phi\gamma_1^\epsilon\psi  dx dt
+\int^T_0 \int_{\Omega} \kappa_1^\epsilon\nabla u_1^\epsilon \cdot \nabla (\phi\gamma_1^\epsilon)\psi  dx dt
-\int^T_0\int_{\Omega} {1 \over \epsilon}Q^\epsilon(u_2^\epsilon-u_1^\epsilon)\phi\gamma_1^\epsilon\psi  dx dt\\
+\int^T_0\int_{\Omega}\div(\kappa_1^\ep \nabla \gamma_1^\ep) \phi u_1^\ep\psi   dx dt
=\int^T_0\int_{\Omega}q \phi\gamma_1^\epsilon\psi  dx  dt
-\frac{1}{\epsilon}\int^T_0\int_\Omega Q(x,\xoe)\phi u_1^\ep\psi dx dt\\
%+\int^T_0\int_\Omega Q(x,\frac{x}{\epsilon})(M_2(x,\frac{x}{\epsilon})-M_1(x,\frac{x}{\epsilon})\phi u_1^\ep\psi dx dt\\
+\epsilon\int^T_0\int_\Omega\div_x \big(\kappa_1(x,\frac{x}{\epsilon})\nabla_xM_1(x,\frac{x}{\epsilon})\big)\phi u_1^\ep\psi dx dt
+\int^T_0\int_\Omega \div_x\big(\kappa_1(x,\frac{x}{\epsilon})\nabla_y M_1 (x,\frac{x}{\epsilon})\big)\phi u_1^\ep\psi dx dt
\\+\int^T_0\int_\Omega\div_y\big(\kappa_1(x,\frac{x}{\epsilon})\nabla_x M_1(x,\frac{x}{\epsilon})\big)\phi u_1^\ep\psi dx dt.\\
%\int_{\Omega}{\mathcal C}_{22}^\epsilon{\partial u_2^\epsilon  \over \partial t} \phi\omega_2^\epsilon  dx
%+ \int_{\Omega} \kappa_2^\epsilon\nabla u_2^\epsilon \cdot \nabla (\phi\omega_2^\epsilon)  dx
%-\int_{\Omega} {1 \over \epsilon}Q^\epsilon(u_1^\epsilon-u_2^\epsilon)\phi\omega_2^\epsilon  dx
%= \int_{\Omega}q \phi\omega_2^\epsilon  dx.
 \end{split}
 \end{equation}
% That is, we have
%  \begin{equation}
%\label{eq:main103*}
% \begin{split}
%\ep\int^T_0\int_{\Omega}{\mathcal C}_{11}^\epsilon{\partial u_1^\epsilon  \over \partial t}\phi M_1  dx\psi dt
%+\ep\int^T_0 \int_{\Omega} \kappa_1^\epsilon\nabla u_1^\epsilon \cdot \nabla (\phi M_1)  dx\psi dt
%-\int^T_0\int_{\Omega} Q^\epsilon(u_2^\epsilon-u_1^\epsilon)\phi M_1  dx \psi dt\\
%+\ep \int^T_0\int_{\Omega}\div(\kappa_1^\ep \nabla M_1) \phi u_1^\ep   dx\psi dt
%=\ep \int^T_0\int_{\Omega}q \phi M_1 \psi dx dt
%%-\frac{1}{\epsilon}\int^T_0\int_\Omega Q^\ep \phi u_1^\ep dx\psi dt\\
%%+\int^T_0\int_\Omega Q(x,\frac{x}{\epsilon})(M_2(x,\frac{x}{\epsilon})-M_1(x,\frac{x}{\epsilon})\phi u_1^\ep\psi dx dt\\
%+\epsilon\int^T_0\int_\Omega\div_x (\kappa_1(x,\frac{x}{\epsilon})\nabla_xM_1(x,\frac{x}{\epsilon}))\phi u_1^\ep dx\psi dt
%+\int^T_0\int_\Omega \div_x(\kappa_1(x,\frac{x}{\epsilon})\nabla_y M_1 (x,\frac{x}{\epsilon}))\phi u_1^\ep dx\psi dt
%%\\+\int^T_0\int_\Omega\div_y(\kappa_1(x,\frac{x}{\epsilon})\nabla_x M_1(x,\frac{x}{\epsilon}))\phi u_1^\ep dx\psi dt.\\
%%\int_{\Omega}{\mathcal C}_{22}^\epsilon{\partial u_2^\epsilon  \over \partial t} \phi\omega_2^\epsilon  dx
%%+ \int_{\Omega} \kappa_2^\epsilon\nabla u_2^\epsilon \cdot \nabla (\phi\omega_2^\epsilon)  dx
%%-\int_{\Omega} {1 \over \epsilon}Q^\epsilon(u_1^\epsilon-u_2^\epsilon)\phi\omega_2^\epsilon  dx
%%= \int_{\Omega}q \phi\omega_2^\epsilon  dx.
%\end{split}
% \end{equation}
We note that on the left hand side of \eqref{eq:main103}, 
\begin{eqnarray*}
\int_0^T\int_\Omega\kappa_1^\ep\nabla u_1^\ep\cdot\nabla(\phi\gamma_1^\ep)\psi dxdt+\int_0^T\int_\Omega\div(\kappa_1^\ep\nabla\gamma_1^\ep)\phi u_1^\ep\psi dxdt=\\\int_0^T\int_\Omega\kappa_1^\ep\nabla u_1^\ep\cdot\nabla\phi\gamma_1^\ep\psi dxdt-\int_0^T\int_\Omega\kappa_1^\ep\nabla\gamma_1^\ep\cdot\nabla\phi u_1^\ep dxdt.
\end{eqnarray*}
 Passing \eqref{eq:main103} to the limit, using Lemma \ref{lemma4}, one obtains
   \begin{equation}
\label{eq:main103''}
 \begin{split}
& -\int_0^T\int_\Omega\left(\int_Y \kappa_1(x,y)\nabla_yM_1(x.y)dy\right)\cdot\nabla\phi(x) u_{10}\psi(t) dxdt\\
&-\int_0^T\int_\Omega\left(\int_YQ(x,y)M_1(x,y)dy\right)(u_{20}-u_{10})\phi(x)\psi(t) dxdt\\
 =
& -\lim_{\ep\to 0}{1\over\ep}\int_0^T\int_\Omega Q(x,{x\over\ep})\phi(x) u_1^\ep\psi(t) dxdt 
+\int_0^T\int_\Omega\div\left(\int_Y\kappa_1(x,y)\nabla_yM_1(x,y)dy\right)\phi(x) u_{10}\psi(t) dxdt\\ &+\int_0^T\int_\Omega\left(\int_Y\div_y(\kappa_1(x,y)\nabla_xM_1(x,y))dy\right)\phi(x) u_{10}\psi(t) dxdt.
 %\displaystyle \lim_{\epsilon \to 0} \bigg(
%\ep\int^T_0 \int_{\Omega} \kappa_1^\epsilon\nabla u_1^\epsilon \cdot \nabla (\phi M_1)  dx\psi dt
%+\ep \int^T_0\int_{\Omega}\div(\kappa_1^\ep \nabla M_1) \phi u_1^\ep   dx\psi dt\\
%-\int^T_0\int_{\Omega} Q^\epsilon(u_2^\epsilon-u_1^\epsilon)\phi M_1  dx \psi dt %\bigg)\\
%\lime_{\ep\to 0{\int_0^T\int_\Omega\kappa_1^\ep\nabla\u_1^\ep\cdot\nabla\phi\gamma_1^\ep\psi dxdt=
%=
%\displaystyle \lim_{\epsilon \to 0} \bigg(
%-\frac{1}{\epsilon}\int^T_0\int_\Omega Q(x,\xoe)\phi u_1^\ep dx\psi dt
%+\int^T_0\int_\Omega Q(x,\frac{x}{\epsilon})(M_2(x,\frac{x}{\epsilon})-M_1(x,\frac{x}{\epsilon})\phi u_1^\ep\psi dx dt\\
%+\int^T_0\int_\Omega \div_x(\kappa_1(x,\frac{x}{\epsilon})(\nabla_y M_1 (x,\frac{x}{\epsilon})))\phi u_1^\ep dx\psi dt
%\bigg).
%\int_{\Omega}{\mathcal C}_{22}^\epsilon{\partial u_2^\epsilon  \over \partial t} \phi\omega_2^\epsilon  dx
%+ \int_{\Omega} \kappa_2^\epsilon\nabla u_2^\epsilon \cdot \nabla (\phi\omega_2^\epsilon)  dx
%-\int_{\Omega} {1 \over \epsilon}Q^\epsilon(u_1^\epsilon-u_2^\epsilon)\phi\omega_2^\epsilon  dx
%= \int_{\Omega}q \phi\omega_2^\epsilon  dx.
 \end{split}
 \end{equation}
 Due to periodicity, the last term on the right hand side equals 0. We thus have
  \begin{equation}
\label{eq:main104}
 \begin{split}
\displaystyle \lim_{\epsilon \to 0} 
&\frac{1}{\ep}\int^T_0 \int_{\Omega} Q(x,\xoe)  u_1^\ep \phi(x) \psi(t)  dx dt\\
%&= \int^T_0\int_{\Omega} \bigg( \int_Y Q(x,y) M_1(x,y)  dy\bigg) (u_{20}-u_{10}) \phi(x)\psi(t) dx dt\\
%&+\int^T_0\int_{\Omega} \div\bigg(\int_Y \kappa_1(x,y) \nabla_y M_1(x,y) dy\bigg) u_{10}\phi(x)\psi(t)  dx dt\\
%&+\int^T_0\int_{\Omega} \bigg( \int_Y \kappa_1(x,y) \nabla_y M_1(x,y) dy \bigg) \cdot\nabla \phi(x)u_{10} \psi(t)  dx dt\\
&=\int^T_0\int_{\Omega} \bigg( \int_Y Q M_1 dy\bigg) (u_{20}-u_{10}) \phi\psi dx dt -\int_0^T\int_\Omega\left(\int_Y\kappa_1\nabla_yM_1dy\right)\cdot\nabla(\phi u_{10})\psi dxdt\\
&\qquad\qquad+\int_0^T\int_\Omega\left(\int_Y\kappa_1\nabla_yM_1dy\right)\cdot\nabla\phi u_{10}\psi dxdt\\
&=\int^T_0\int_{\Omega} \bigg( \int_Y Q M_1  dy\bigg) (u_{20}-u_{10}) \phi\psi dx dt-\int^T_0\int_{\Omega} \bigg( \int_Y \kappa_1 \nabla_y M_1 dy \bigg) \cdot\nabla u_{10} \phi\psi  dx dt.
 \end{split}
 \end{equation}
 
Similarly, we obtain 
 \begin{equation}
\label{eq:main105}
 \begin{split}
\displaystyle \lim_{\epsilon \to 0}&
\frac{1}{\ep}\int^T_0 \int_{\Omega} Q(x,\xoe) u_2^\ep \phi(x) \psi(t)  dx dt\\
&=\int^T_0\int_{\Omega} \bigg( \int_Y Q M_2  dy\bigg) (u_{10}-u_{20}) \phi\psi dx dt-\int^T_0\int_{\Omega} \bigg( \int_Y \kappa_2 \nabla_y M_2  dy \bigg) \cdot\nabla u_{20}\phi\psi  dx dt.\\
 \end{split}
 \end{equation}
 Thus
 \begin{equation}
\label{eq:main106'}
 \begin{split}
\displaystyle \lim_{\epsilon \to 0} 
\frac{1}{\ep} &\int^T_0  \int_{\Omega}  Q(x,\xoe) (u_2^\ep - u_1^\ep) \phi(x) \psi(t)  dx dt\\
= &- \int^T_0\int_{\Omega} \bigg( \int_Y Q(M_1 + M_2) dy\bigg) (u_{20}-u_{10}) \phi\psi dx dt\\
&-\int^T_0\int_{\Omega} \left(\int_Y \kappa_2 \nabla_y M_2 dy\right)\cdot \nabla u_{20}\phi\psi  dx dt+\int^T_0\int_{\Omega}  \left(\int_Y \kappa_1 \nabla_y M_1 dy\right)\cdot \nabla u_{10}\phi\psi  dx dt.\\
%= - \int^T_0\int_{\Omega} \bigg( \int_Y Q(x,\xoe)(M_1(x,\xoe)+ M_2(x,\xoe))  dy\bigg) (u_{20}-u_{10}) \phi(x)\psi(t) dx dt\\
%-\int^T_0\int_{\Omega} \bigg(\int_Y \kappa_1(x,\xoe) \nabla_y M_1(x,\xoe) dy\bigg) \nabla (u_{20}\phi(x))  dx \psi(t)  dt\\
%+\int^T_0\int_{\Omega}  \bigg(\int_Y \kappa_1(x,\xoe) \nabla_y M_1(x,\xoe) dy\bigg) \nabla (u_{10}\phi(x))  dx \psi(t) dt\\
 %\end{split}
 %\end{equation}
 %\jrp{
%Thus,
%\begin{equation}
%\label{eq:main106'}
 %\begin{split}
%\displaystyle \lim_{\epsilon \to 0} 
%\frac{1}{\ep} &\int^T_0  \int_{\Omega}  Q(x,\xoe) (u_2^\ep - u_1^\ep) \phi(x) \psi(t)  dx dt\\
 =& - \int^T_0\int_{\Omega} \bigg( \int_Y Q(M_1+ M_2) dy\bigg) (u_{20}-u_{10}) \phi(x)\psi(t) dx dt\\
&-\int^T_0\int_{\Omega} \bigg(\int_Y \kappa_2e^i \cdot \nabla_y M_2 dy\bigg) \frac{\partial u_{20}}{\partial x_i}\phi  dx \psi  dt
+\int^T_0\int_{\Omega}  \bigg(\int_Y \kappa_1e^i \cdot  \nabla_y M_1 dy\bigg) \frac{\partial u_{10}}{\partial x_i}\phi  dx \psi dt\\
% \end{split}
% \end{equation}
%From the cell problem (\ref{eq:cell} (a),(c)), \vhh{the right hand side of \eqref{eq:main106'} equals}
% \begin{equation}
%\label{eq:main106}
% \begin{split}
% \displaystyle \lim_{\epsilon \to 0} 
%\frac{1}{\ep} &\int^T_0  \int_{\Omega}  Q(x,\xoe) (u_2^\ep - u_1^\ep) \phi(x) \psi(t)  dx dt\\
= &- \int^T_0\int_{\Omega} \bigg( \int_Y Q(M_1+ M_2) dy\bigg) (u_{20}-u_{10})\phi\psi dx dt\\
&+\int^T_0\int_{\Omega} \bigg(\int_Y \kappa_2 \nabla_y N^i_2 \cdot \nabla_y M_2 dy\bigg) \frac{\partial u_{20}}{\partial x_i}\phi\psi  dx \psi  dt
-\int^T_0\int_{\Omega}  \bigg(\int_Y \kappa_1\nabla_y N^i_1  \cdot \nabla_y M_1 dy\bigg) \frac{\partial u_{10}}{\partial x_i}\phi\psi  dx dt
\end{split}
\end{equation}
where we have used cell problems (\ref{eq:cell} (a),(c)).
Using cell problems (\ref{eq:cell} (b),(d)), we have 
 \begin{equation}
 \label{eq:main106}
\begin{split}
\lim_{\ep\to0}\frac{1}{\ep} &\int^T_0  \int_{\Omega}  Q(x,\xoe) (u_2^\ep - u_1^\ep) \phi \psi  dx dt\\
&=- \int^T_0\int_{\Omega} \bigg( \int_Y Q(M_1+ M_2) dy\bigg) (u_{20}-u_{10}) \phi\psi dx dt\\
&+\int^T_0\int_{\Omega} \bigg(\int_Y QN_2^i dy\bigg) \frac{\partial u_{20}}{\partial x_i}\phi \psi dx  dt
-\int^T_0\int_{\Omega}  \bigg(\int_YQN_1^i dy\bigg) \frac{\partial u_{10}}{\partial x_i} \phi\psi  dx dt.\\
%+ \int^T_0\int_{\Omega} \div_x \big[\int_Y \kappa_1(x,\xoe) \nabla_y M_1(x,\xoe) dy(u_{20}-u_{10})\big] \phi(x)  dx \psi(t) dt
 \end{split}
 \end{equation}
 
 %%%%%%%%%%%%%%%%%%%%%%%%%%%%%%%%%%%%%%%%%%%%%%%%%%%%%%
We are now ready to prove Theorem \ref{weakconv}.

{\it Proof of Theorem \ref{weakconv}}

%We now prove Theorem \ref{weakconv}.
From \eqref{eq:main65}
\begin{equation}
\label{eq:main107}
 \begin{split}
\int_0^T\int_{\Omega}{\mathcal C}_{11}^\epsilon{\partial u_1^\epsilon  \over \partial t}\phi \psi  dx   dt
+ \int_0^T\int_{\Omega} \kappa_1^\epsilon\nabla u_1^\epsilon \cdot \nabla \phi \psi dx  dt
- {1 \over \ep }\int_0^T\int_{\Omega} Q^\epsilon (u_2^\ep -u_1^\ep) \phi\psi  dx  dt 
 = \int_0^T\int_{\Omega}q \phi \psi dx  dt.
 \end{split}
 \end{equation}
for all $\phi\in C^\infty_0 (\Omega)$ and $\psi\in C^\infty_0((0,T))$.
Passing to the limit, from (\ref{eq:main95-1}), (\ref{eq:main106}), Lemma \ref{lemma4}, we have
\begin{equation}
 \label{eq:main108}
 \begin{split}
\int_0^T\int_{\Omega}\int_Y& {\mathcal C}_{11} dy{\partial u_{10} \over \partial t} \phi \psi dx dt\\
&+ \int_0^T\int_{\Omega}\kappa_1^*\nabla u_{10}\cdot\nabla \phi\psi dx dt
+\int_0^T\int_{\Omega}\left(\int_Y \kappa_1 \nabla_y M_1dy\right)\cdot\nabla\phi(u_{20}-u_{10})\psi dx dt
\\ &- \int_0^T\int_{\Omega}\left(\left(\int_Y QN^i_2  dy\right)\frac{\partial u_{20}}{\partial x_i} -\left( \int_Y QN^i_1  dy\right)\frac{\partial u_{10}}{\partial x_i}\right) \phi\psi dx dt\\
&+ \int_0^T\int_{\Omega}\left(\int_YQ(M_1+M_2)  dy\right)(u_{20}-u_{10}) \phi\psi dx dt 
=\int_0^T\int_{\Omega} q \phi\psi dx dt.\\
\end{split}
\end{equation}
Similarly, we  derive
\begin{equation}
 \label{eq:main109}
 \begin{split}
\int_0^T\int_{\Omega}\int_Y& {\mathcal C}_{22} dy{\partial u_{20} \over \partial t} \phi dx \psi dt\\
&+\int_0^T\int_{\Omega}\kappa_2^*\nabla u_{20}\cdot\nabla \phi\psi dx dt
+\int_0^T\int_{\Omega}\left(\int_Y \kappa_2 \nabla_y M_2dy\right)\cdot\nabla\phi(u_{10}-u_{20})\psi dx dt
\\ &-\int_0^T\int_{\Omega}\left(\left(\int_Y QN^i_1  dy\right)\frac{\partial u_{10}}{\partial x_i} - \left(\int_Y QN^i_2  dy\right)\frac{\partial u_{20}}{\partial x_i}\right)\phi\psi dx dt\\
&+\int_0^T\int_{\Omega}\left(\int_YQ(M_1+M_2)  dy\right)(u_{10}-u_{20}) \phi\psi dx dt.
=\int_0^T\int_{\Omega} q \phi\psi dx dt,\\
\end{split}
\end{equation}
where 
\begin{equation}
 \label{eq:main110}
 \begin{split}
\kappa^*_{1ij}(x) =  \int_Y  \kappa_1 (x,y)(\delta_{ij} + {\partial  N^j_1(x,y)\over \partial y_i})  dy,\ \ 
\kappa^*_{2ij}(x) =  \int_Y  \kappa_2 (x,y)(\delta_{ij} + {\partial  N^j_2(x,y)\over \partial y_i})  dy.
\end{split}
\end{equation}

We now prove the initial condition of $u_{10},\ u_{20}$.
From (\ref{eq:main68'*'}) and (\ref{eq:main68'**}) and (\ref{eq:unibound}), we deduce that ${\mathcal C}_{11}^\epsilon{\partial \hat{u}_1^\epsilon  \over \partial t} $ is also bounded in $L^2(0,T;V')$, where $\hat{u}_1^\ep = u_1^\ep e^{-\lambda t}$ and thus, ${\mathcal C}_{11}^\epsilon{\partial u_1^\epsilon  \over \partial t} $ is bounded in $L^2(0,T;V')$. Similarly,
${\mathcal C}_{22}^\epsilon{\partial u_2^\epsilon  \over \partial t}$ is bounded in $L^2(0,T;V')$.
Let $\psi(t,x) \in {\mathcal C}_0^\infty (0,T;V)$, i.e. $\psi(0,x) = \psi(T,x) = 0$. As $\ep \rightarrow 0$, we have
\begin{equation}
\begin{split}
\int_0^T \int_\Omega {\mathcal C}_{11}^\epsilon{\partial u_1^\epsilon  \over \partial t}  \psi  dx  dt 
= - \int_0^T \int_\Omega {\mathcal C}_{11}^\epsilon  u_1^\epsilon \frac{\partial \psi}{\partial t}   dx  dt 
\rightarrow - \int_0^T \int_\Omega \langle {\mathcal C}_{11} \rangle u_{10} \frac{\partial \psi}{\partial t}  dx  dt = \int_0^T \int_\Omega \langle {\mathcal C}_{11} \rangle \frac{\partial u_{10}}{\partial t} \psi  dx  dt
\end{split}
\end{equation}
where $\langle\cdot\rangle$ denotes the integral average over $Y$. Note that we used Lemma \ref{lemma4}.
This shows that the weak limit of ${\mathcal C}_{11}^\epsilon{\partial u_1^\epsilon  \over \partial t} $ in $L^2(0,T;V')$ is $\langle {\mathcal C}_{11} \rangle  \frac{\partial u_{10}}{\partial t}$. Now we choose $\psi \in {\mathcal C}^\infty (0,T;V) $ so that $\psi(T,x)=0$. Then

\begin{equation}
\begin{split}
\int_0^T \int_\Omega {\mathcal C}_{11}^\epsilon{\partial u_1^\epsilon  \over \partial t}\psi  dx  dt &= - \int_0^T \int_\Omega {\mathcal C}_{11}^\epsilon  u_1^\epsilon \frac{\partial \psi}{\partial t}  dx  dt + \int_\Omega {\mathcal C}_{11}^\epsilon  u_1^\epsilon(0,x) \psi(0,x) dx \\
&\rightarrow -\int_0^T \int_\Omega \langle {\mathcal C}_{11} \rangle u_{10} \frac{\partial \psi}{\partial t}  dx  dt + \int_\Omega \langle {\mathcal C}_{11} \rangle g_1 \psi(0,x)  dx.
\end{split}
\end{equation}
On the other hand
\begin{equation}
\begin{split}
\int_0^T \int_\Omega {\mathcal C}_{11}^\epsilon{\partial u_1^\epsilon  \over \partial t} \psi  dx  dt 
\rightarrow  \int_0^T \int_\Omega \langle {\mathcal C}_{11} \rangle  \frac{\partial u_{10}}{\partial t} \psi  dx  dt\\
=- \int_0^T \int_\Omega \langle {\mathcal C}_{11} \rangle u_{10} \frac{\partial \psi}{\partial t}  dx  dt 
+ \int_\Omega \langle {\mathcal C}_{11} \rangle  u_{10}(0,x) \psi(0,x)  dx.
\end{split}
\end{equation}
This shows that $\langle {\mathcal C}_{11} \rangle u_{10}(0,x) = \langle {\mathcal C}_{11} \rangle g_1(x)$. i.e. the initial condition of $u_{10}$ is $u_{10}(0,x) = g_1(x)$.
 Similarly, we have initial condition $u_{20}(0,x) =  g_2(x)$.
%%%%%%%%%%%%%%%%%%%%%%%%%%%%%%%%% Hom Error %%%%%%%%%%%%%%%%%%%%%%%%%%
\section{Homogenization error}
We prove Theorem \ref{hom_error} in this section. Let
\beq
\begin{split}
\label{uiiep}
u_{11}^\ep(t,x) = u_{10}(t,x) + \ep u_{11}(t,x,\xoe), \
u_{21}^\ep(t,x) = u_{20}(t,x) + \ep u_{21}(t,x,\xoe).
\end{split}
\eeq
Using
%\eqref{uiiep} and the definition of $u_{ii}$ 
\eqref{eq:main7} 
%with $y = \xoe$, 
we have
\beq
\label{eq:hom_error1-1}
\bsp
&\text{div}(\kappa_1^\epsilon(x)\nabla u_{11}^\epsilon(t,x)) + {1 \over \epsilon}Q^\epsilon(x)(u_{21}^\epsilon(t,x)-u_{11}^\epsilon(t,x))
\\&=
\div( \kappa_1^\ep \nabla u_{10}) + \ep\div(\kappa_1^\ep \nabla_x u_{11}) +\div (\kappa_1^\ep \nabla_y N^i_1(x,\xoe) \frac{\partial u_{10}}{\partial x_i}) 
 + \div( \kappa_1^\ep \nabla_y M_1(x,\xoe)(u_{20}-u_{10}))\\
&+\frac{1}{\ep} Q^\epsilon(u_{20}-u_{10}) 
 +Q^\epsilon (N^i_2(x,\xoe)\frac{\partial u_{20}}{\partial x_i}-N^i_1(x,\xoe) \frac{\partial u_{10}}{\partial x_i})
 +Q^\epsilon(M_2(x,\xoe)+M_1(x,\xoe))(u_{10}-u_{20})\\
& = \div( \kappa_1^\ep \nabla u_{10}) + \ep\div(\kappa_1^\ep \nabla_x u_{11}) 
 +\div (\kappa_1^\ep\nabla_y N^i_1(x,\xoe) \frac{\partial u_{10}}{\partial x_i}) 
 + \div(\kappa_1^\ep \nabla_y M_1(x,\xoe) (u_{20}-u_{10}))\\
&+\div (\mathcal{Q}(x,\xoe) (u_{20}-u_{10}) )- \div_x(\mathcal{Q} (x,\xoe) (u_{20}-u_{10}) )
 +Q(x,\xoe) (N^i_2(x,\xoe) \frac{\partial u_{20}}{\partial x_i}-N^i_1(x,\xoe)\frac{\partial u_{10}}{\partial x_i})\\
& +Q(x,\xoe)(M_2(x,\xoe)+M_1(x,\xoe))(u_{10}-u_{20})\\
%\end{split}
%\eeq
%Thus, we deduce
%\beq
%\label{eq:hom_error1}
%\bsp
%\text{div}(\kappa_1^\epsilon(x)\nabla u_{11}^\epsilon(t,x)) + {1 \over \epsilon}Q^\epsilon(x)(u_{21}^\epsilon(t,x)-u_{11}^\epsilon(t,x))\\
& =\div( \kappa_1^\ep \nabla u_{10}) 
 + \ep\div(\kappa_1^\ep \nabla_x u_{11}) 
 +\div (\kappa_1^\ep \nabla_y N^i_1(x,\xoe) \frac{\partial u_{10}}{\partial x_i}) 
 + \div( \kappa_1^\ep \nabla_y M_1(x,\xoe) (u_{20}-u_{10}))\\
& +\div (\mathcal{Q} (x,\xoe)(u_{20}-u_{10}) )- \div_x(\mathcal{Q}(x,\xoe) (u_{20}-u_{10}) )
 -\div( \int_Y \mathcal{Q}(x,y) dy (u_{20}-u_{10})) \\
 &+ \div(\int_Y \mathcal{Q}(x,y) dy (u_{20}-u_{10}))+Q(x,\xoe) (N^i_2(x,\xoe) \frac{\partial u_{20}}{\partial x_i}-N^i_1(x,\xoe) \frac{\partial u_{10}}{\partial x_i})\\
& +Q(x,\xoe)(M_2(x,\xoe)+M_1(x,\xoe))(u_{10}-u_{20}).\\
\end{split}
\eeq
We let $F(t,x,y)$ be defined as
\beq
\bsp
F(t,x,y) = \kappa_1(x,y) \nabla u_{10}(t,x)+ \kappa_1(x,y) \nabla_y N^i_1(x,y) \frac{\partial u_{10}(t,x)}{\partial x_i}
 + \kappa_1(x,y) \nabla_y M_1(x,y) (u_{20}(t,x)-u_{10}(t,x))\\
 +\mathcal{Q} (x,y) (u_{20}(t,x)-u_{10}(t,x))  - \int_Y \mathcal{Q}(x,y) dy (u_{20}(t,x)-u_{10}(t,x))\\
 - \left(\int_Y \kappa_1(x,y)dy \nabla u_{10}(t,x)+ \int_Y \kappa_1 (x,y)\nabla_y N^i_1(x,y) dy \frac{\partial u_{10}(t,x)}{\partial x_i}\right.\\
 \left.+ \int_Y \kappa_1(x,y) \nabla_y M_1(x,y) dy (u_{20}(t,x)-u_{10}(t,x))\right).
\end{split}
\eeq
We let
\beq
\bsp
&G (t,x,y) = - \div_x(\mathcal{Q}(x,y) (u_{20}-u_{10}) ) + \div \big(\int_Y \mathcal{Q}(x,y) dy (u_{20}-u_{10})\big)
 +Q(x,y) \big(N^i_2(x,y) \frac{\partial u_{20}}{\partial x_i}-N^i_1(x,y) \frac{\partial u_{10}}{\partial x_i}\big)\\
 &+Q(x,y)(M_2(x,y)+M_1(x,u))(u_{10}-u_{20})
 - \bigg( \int_Y Q(x,y) N^i_2(x,y)dy \frac{\partial u_{20}}{\partial x_i}-\int_YQ(x,y)N^i_1(x,y)dy \frac{\partial u_{10}}{\partial x_i}\bigg) \\
& -\int_Y Q(x,y)(M_2(x,y)+M_1(x,y))(u_{10}-u_{20}) dy.
\end{split}
\eeq
Note that from (\ref{eq:cell}), we deduce $\div_y F(t,x,y) = 0$. Further,
we have $\int_Y F_i(t,x,y) dy = 0$, $i=1,\ldots,d$. From the hypothesis of the theorem, $F_i(t,x,y) \in C(0,T;C^1(\bar{\Omega};C(\bar{Y})))$.
Thus, from \cite{jikov2012homogenization},  there are functions $\alpha_{ij}(t,x,y) \in C(0,T;C^1(\bar{\Omega};C^1(\bar{Y})))$ such that
\beq
\alpha_{ij} = - \alpha_{ji}\ \textrm{and} \ F_i(t,x,y) = \frac{\partial}{\partial y_j}\alpha_{ij}(t,x,y),
\eeq
for $i,j=1,\ldots,d$.
From this, we have
\beq
F_i(t,x,\xoe) = \ep\frac{ d}{ d x_j}\alpha_{ij}(t,x,\xoe)-\ep\frac{\partial}{\partial x_j}\alpha_{ij}(t,x,\xoe),
\eeq
where $\frac{ d}{ d x_j}$ is the total partial derivative with respect to $x_j$ of a function of $t$ and $x$.
Then for any $\phi(x) \in V$, we have
\beq
\label{eq:main128}
\bsp
\int_\Omega F_i(t,x,\xoe) \frac{\partial}{\partial x_i} \phi(x) dx
=\int_\Omega \big(\ep\frac{ d}{ d x_j}\alpha_{ij}(t,x,\xoe)-\ep\frac{\partial}{\partial x_j}\alpha_{ij}(t,x,\xoe)\big) \frac{\partial}{\partial x_i} \phi(x) dx\\
=-\ep \int_\Omega \alpha_{ij}(t,x,\xoe) \frac{\partial^2 \phi(x)}{\partial x_j \partial x_i} dx
-\ep \int_\Omega \frac{\partial}{\partial x_j}\alpha_{ij}(t,x,\xoe) \frac{\partial}{\partial x_i} \phi(x) dx
= -\ep \int_\Omega \frac{\partial}{\partial x_j}\alpha_{ij}(t,x,\xoe) \frac{\partial}{\partial x_i} \phi(x) dx.
\end{split}
\eeq
As $\int_Y G(t,x,y) dy = 0$, there exists a vector function $\mathcal{G}\in$ $C(0,T;C^1(\bar{\Omega};C^1(\bar{Y})))$ which is $Y$-periodic with respect to $y$
such that 
$\div_y \mathcal{G} = G$. Thus for any $\phi(x) \in V$, we have
\beq
\label{eq:main129}
\bsp
\int_\Omega G(t,x,\xoe) \phi dx
= \int_\Omega \div_y \mathcal{G}(t,x,\xoe)\phi dx
=\ep \int_\Omega \div \mathcal{G}(t,x,\xoe)\phi dx - \ep \int_\Omega \div_x \mathcal{G}(t,x,\xoe)\phi dx\\
= -\ep \int_\Omega \mathcal{G}(t,x,\xoe) \cdot \nabla \phi dx - \ep \int_\Omega \div_x \mathcal{G}(t,x,\xoe)\phi dx.
\end{split}
\eeq
%We observe $u_{11} \in C(0,T; C^1(\bar{\Omega};C^1(\bar{Y})))$ and
%\beq
%\label{eq:main130}
%\ep \int_\Omega \div(\kappa_1 \nabla_x u_{11}) \phi(x) dx = - \ep \int_\Omega \kappa_1 \nabla_x u_{11} \cdot \nabla \phi(x) dx.
%\eeq
%It can be also shown that
From \eqref{uiiep}, we have
\beq
\label{eq:main130*}
\norm{{\mathcal C}_{11}^\epsilon{\partial u_{11}^\epsilon  \over \partial t}(t) 
- \int_Y {\mathcal C}_{11} dy{\partial u_{10} \over \partial t}(t) }_{V'} \leq c\ep
\eeq
where $c$ is independent of $t$.
From (\ref{eq:hom_error1-1}), (\ref{eq:main128}), (\ref{eq:main129}) and (\ref{eq:main130*}), we have
\beq
\label{eq:main131}
\bsp
&\left\|\bigg({\mathcal C}_{11}^\epsilon{\partial u_{11}^\epsilon  \over \partial t}(t)
-\text{div}(\kappa_1^\epsilon(x)\nabla u_{11}^\epsilon(t)) - {1 \over \epsilon}Q^\epsilon(u_{21}^\epsilon(t)-u_{11}^\epsilon(t))\bigg)\right.\\
&\left.-\bigg(\int_Y {\mathcal C}_{11} dy{\partial u_{10} \over \partial t}(t)
- \div (\kappa_1^*\nabla u_{10}(t)) 
-\div \big(\int_Y \kappa_1 \nabla_y M_1  dy(u_{20}(t)-u_{10}(t))\big)\right.\\
&\left.- \big(\int_Y QN^i_2  dy\frac{\partial u_{20}}{\partial x_i}(t) - \int_Y QN^i_1  dy\frac{\partial u_{10}}{\partial x_i}(t)\big) + \int_YQ(M_1+M_2)  dy(u_{20}(t)-u_{10}(t))
\bigg)\right\|_{V'}
\leq c\ep.
\end{split}
\eeq
Let $\tau^\ep \in \mathcal{D} (\Omega)$ be such that 
%$\tau^\ep (x) = 1$ outside an $\ep$ neighborhood of $\partial \Omega$
%and $\ep |\nabla_x \tau^\ep(x)| \leq C $ where $C$ is independent of $\ep$.
\beqas
\bsp
\tau^\ep (x) = 0 \  \enspace \textrm{if} \ \  d(x,\partial \Omega) \leq \ep, \ \
\tau^\ep (x) = 1 \   \enspace \textrm{if} \ \  d(x,\partial \Omega) \geq 2\ep, \
\ep |\nabla \tau^\ep(x)| \leq C, 
\end{split}
\eeqas
 where $C$ is independent of $\ep$. We define the functions
 \beq
 \begin{split}
 \omega_{11}^\ep (t,x)= u_{10}(t,x) + \ep \tau^\ep(x) u_{11}(t,x,\xoe), \ \ 
 \omega_{21}^\ep (t,x)= u_{20}(t,x) + \ep \tau^\ep(x) u_{21}(t,x,\xoe).
 \end{split}
 \eeq
% Note again that we have
Using the smoothness asumptions of the theorem, we have
 \beq
 \label{eq:main131*}
 \bsp
 \nabla (u_{11}^\eps(t,x) - \omega_{11}^\eps(t,x)) = -\eps \nabla \tau^\eps(x) u_{11} (t,x,\xoe) + \eps(1-\tau^\eps(x)) \nabla_x u_{11}(t,x,\xoe) 
 + (1-\tau^\eps(x)) \nabla_y u_{11} (t,x,\xoe), \\
 \nabla (u_{21}^\eps(t,x) - \omega_{21}^\eps(t,x)) = -\eps \nabla \tau^\eps(x) u_{21} (t,x,\xoe) + \eps(1-\tau^\eps(x)) \nabla_x u_{21}(t,x,\xoe) 
 + (1-\tau^\eps(x)) \nabla_y u_{21} (t,x,\xoe).
 \end{split}
 \eeq
 It follows from (\ref{eq:main131*}) that
 \beq
 \label{eq:main132}
 \bsp
 \|u_{11}^\eps(t) - \omega_{11}^\eps(t)\|_{H^1(\Omega)} \leq c \eps^{\frac{1}{2}},\ \  \|u_{21}^\eps(t) - \omega_{21}^\eps(t)\|_{H^1(\Omega)} \leq c \eps^{\frac{1}{2}}
  \end{split}
 \eeq
where the constant $c$ is independent of $t$.
% Note that $\mathcal{Q}\in C^1(\bar{\Omega};C(\bar{Y}))^2$ and ${\mathcal C}_{ii}^\epsilon(x), \kappa_i^\ep(x)$ are continuous. 
From (\ref{eq:main132}), we have
 \beq 
 \bsp
& \int_\Omega {\mathcal C}_{11}^\epsilon(x){\partial (u_{11}^\epsilon(t,x)- \omega_{11}^\epsilon(t,x))  \over \partial t} \phi(x)dx
 +\int_\Omega \kappa_1^\epsilon(x)\nabla (u_{11}^\epsilon(t,x) - \omega_{11}^\epsilon(t,x)) \cdot \nabla  \phi(x) dx \\
 &-\int_\Omega \mathcal{Q}(x,\xoe) ((u_{11}^\epsilon(t,x) - \omega_{11}^\epsilon(t,x))-(u_{21}^\epsilon(t,x) - \omega_{21}^\epsilon(t,x))) \nabla \phi(x) dx\\
&  -\int_\Omega \mathcal{Q}(x,\xoe) ((\nabla u_{11}^\epsilon(t,x) - \nabla\omega_{11}^\epsilon(t,x))-(\nabla u_{21}^\epsilon(t,x) - \nabla\omega_{21}^\epsilon(t,x)))  \phi(x) dx\\
 &-\int_\Omega \div_x \mathcal{Q}(x,\xoe) ((u_{11}^\epsilon(t,x) - \omega_{11}^\epsilon(t,x))-(u_{21}^\epsilon(t,x) - \omega_{21}^\epsilon(t,x))) \phi(x) dx\\
 &\leq c \left( \norm{{\partial (u_{11}^\epsilon(t)- \omega_{11}^\epsilon(t))  \over \partial t}}_{H^1(\Omega)}
 + \norm{(u_{11}^\epsilon(t) - \omega_{11}^\epsilon(t))}_{H^1(\Omega)}
 + \norm{(u_{21}^\epsilon(t) - \omega_{21}^\epsilon(t))}_{H^1(\Omega)}\right) \norm{\phi}_V
 \leq c \ep^{1 \over 2} \norm{\phi}_V
 \end{split}
 \eeq
 for all $\phi \in V$, where $c>0$ is independent of $t$. 
 %Then after taking $\int_0^T \cdot dt$, we obtain
 Then we obtain
\beq
\label{eq:a}
\bsp
\Bigg\|\bigg({\mathcal C}_{11}^\epsilon(x){\partial u_{11}^\epsilon(t,x)  \over \partial t}
&-\text{div}(\kappa_1^\epsilon(x)\nabla u_{11}^\epsilon(t,x)) - {1 \over \epsilon}Q^\epsilon(x)(u_{21}^\epsilon(t,x)-u_{11}^\epsilon(t,x))\bigg)\\
-\bigg({\mathcal C}_{11}^\epsilon(x){\partial \omega_{11}^\epsilon(t,x)  \over \partial t}
&-\text{div}(\kappa_1^\epsilon(x)\nabla \omega_{11}^\epsilon(t,x)) - {1 \over \epsilon}Q^\epsilon(x)(\omega_{21}^\epsilon(t,x)-\omega_{11}^\epsilon(t,x))\bigg)\Bigg\|_{{V'}} \leq c \ep^{1 \over 2}.
\end{split}
\eeq
From \eqref{eq:main1}, \eqref{eq:main9}, \eqref{eq:main131} and \eqref{eq:a}, we obtain
\beq
\label{eq:ep12}
\bsp
\Bigg\|\bigg({\mathcal C}_{11}^\epsilon(x){\partial u_{1}^\epsilon(t,x)  \over \partial t}
&-\text{div}(\kappa_1^\epsilon(x)\nabla u_{1}^\epsilon(t,x)) - {1 \over \epsilon}Q^\epsilon(x)(u_{2}^\epsilon(t,x)-u_{1}^\epsilon(t,x))\bigg)\\
-\bigg({\mathcal C}_{11}^\epsilon(x){\partial \omega_{11}^\epsilon(t,x)  \over \partial t}
&-\text{div}(\kappa_1^\epsilon(x)\nabla \omega_{11}^\epsilon(t,x)) - {1 \over \epsilon}Q^\epsilon(x)(\omega_{21}^\epsilon(t,x)-\omega_{11}^\epsilon(t,x))\bigg)\Bigg\|_{{V'}} \leq c \ep^{1 \over 2}
\end{split}
\eeq
where $c$ is independent of $t$.
Similarly,
\beq
\label{eq:ep12-1}
\bsp
\Bigg\|\bigg({\mathcal C}_{22}^\epsilon(x){\partial u_{2}^\epsilon(t,x)  \over \partial t}
&-\text{div}(\kappa_2^\epsilon(x)\nabla u_{2}^\epsilon(t,x)) - {1 \over \epsilon}Q^\epsilon(x)(u_{1}^\epsilon(t,x)-u_{2}^\epsilon(t,x))\bigg)\\
-\bigg({\mathcal C}_{22}^\epsilon(x){\partial \omega_{21}^\epsilon(t,x)  \over \partial t}
&-\text{div}(\kappa_2^\epsilon(x)\nabla \omega_{21}^\epsilon(t,x)) - {1 \over \epsilon}Q^\epsilon(x)(\omega_{11}^\epsilon(t,x)-\omega_{21}^\epsilon(t,x))\bigg)\Bigg\|_{{V'}} \leq c \ep^{1 \over 2}.
\end{split}
\eeq
%We let $\hat{u}_i^\ep(t,x) = u_i^\ep (t,x)e^{-\lambda t},\ \hat{\omega}_{i1}^\ep(t,x) = \omega_{i1}^\ep (t,x)e^{-\lambda t}$. From \eqref{eq:ep12} and \eqref{eq:ep12-1} we obtain
%\beq
%\label{eq:ep12-2}
%\bsp
%||\bigg({\mathcal C}_{11}^\epsilon(x){\partial \hat{u}_1^\ep(t,x)  \over \partial t}
%&-\text{div}(\kappa_1^\epsilon(x)\nabla \hat{u}_1^\ep(t,x)) - {1 \over \epsilon}Q^\epsilon(x)(\hat{u}_2^\ep(t,x)-\hat{u}_1^\ep(t,x))\bigg)\\
%-\bigg({\mathcal C}_{11}^\epsilon(x){\partial \hat{\omega}_{11}^\ep(t,x)  \over \partial t}
%&-\text{div}(\kappa_1^\epsilon(x)\nabla \hat{\omega}_{11}^\ep(t,x)) - {1 \over \epsilon}Q^\epsilon(x)(\hat{\omega}_{21}^\ep(t,x)-\hat{\omega}_{11}^\ep(t,x))\bigg)||_{\todo{H^{-1}(\Omega)}}\leq c e^{-\lambda t} \ep^{1 \over 2} \leq c \ep^{1 \over 2}.
%\end{split}
%\eeq
%Similarly,
%\beq
%\label{eq:ep12-3}
%\bsp
%||\bigg({\mathcal C}_{22}^\epsilon(x){\partial \hat{u}_2^\ep(t,x)  \over \partial t}
%&-\text{div}(\kappa_2^\epsilon(x)\nabla \hat{u}_2^\ep(t,x)) - {1 \over \epsilon}Q^\epsilon(x)(\hat{u}_1^\ep(t,x)-\hat{u}_2^\ep(t,x))\bigg)\\
%-\bigg({\mathcal C}_{22}^\epsilon(x){\partial \hat{\omega}_{21}^\ep(t,x)  \over \partial t}
%&-\text{div}(\kappa_2^\epsilon(x)\nabla \hat{\omega}_{21}^\ep(t,x)) - {1 \over \epsilon}Q^\epsilon(x)(\hat{\omega}_{11}^\ep(t,x)-\hat{\omega}_{21}^\ep(t,x))\bigg)||_{\todo{H^{-1}(\Omega)}}\leq c e^{-\lambda t} \ep^{1 \over 2} \leq c \ep^{1 \over 2}.
%\end{split}
%\eeq
%\todo{I think I didn't use \eqref{eq:ep12-2},\eqref{eq:ep12-2}. I used \eqref{eq:ep12},\eqref{eq:ep12-1}.}\\
Let $\lambda>0$. Let $\hat{u}_i^\ep(t,x) = u_i^\ep (t,x)e^{-\lambda t},\ \hat{\omega}_{i1}^\ep(t,x) = \omega_{i1}^\ep (t,x)e^{-\lambda t}$ for $i=1,2$.
From \eqref{eq:ep12} and \eqref{eq:ep12-1}, we deduce 
%the following with the test functions $\hat{u}_{i}^\ep(t)e^{-\lambda t} -\hat{\omega}_{i1}^\ep(t)e^{-\lambda t} \in V$.
\beq
\label{eq:lamb1}
\bsp
 \int_\Omega {\mathcal C}_{11}^\epsilon {\partial   \over \partial t} (u_{1}^\epsilon -\omega_{11}^\ep) (\hat{u}_{1}^\ep  -\hat{\omega}_{11}^\ep)dx 
+ \int_\Omega \kappa_1^\epsilon {\big(\nabla u_{1}^\ep - \nabla \omega_{11}^\epsilon\big)\cdot(\nabla \hat{u}_{1}^\ep -\nabla \hat{\omega}_{11}^\ep)} dx \\
+ \int_\Omega {\mathcal C}_{22}^\epsilon {\partial   \over \partial t} (u_{2}^\epsilon -\omega_{21}^\ep) (\hat{u}_{2}^\ep -\hat{\omega}_{21}^\ep ) dx 
+ \int_\Omega \kappa_2^\epsilon\big(\nabla u_{2}^\ep - \nabla \omega_{21}^\epsilon\big)\cdot(\nabla \hat{u}_{2}^\ep  -\nabla \hat{\omega}_{21}^\ep )dx \\
- \int_\Omega \mathcal{Q}(x,\xoe)\cdot\big((\nabla u_1^\ep -\nabla\omega_{11}^\ep)-(\nabla u_2^\ep-\nabla \omega_{21}^\ep)\big) 
\big((\hat{u}_{1}^\ep -\hat{\omega}_{11}^\ep)-(\hat{u}_{2}^\ep  -\hat{\omega}_{21}^\ep )\big) dx \\
- \int_\Omega \mathcal{Q}(x,\xoe)\cdot\big((\nabla\hat{u}_{1}^\ep  -\nabla\hat{\omega}_{11}^\ep )-(\nabla\hat{u}_{2}^\ep  -\nabla\hat{\omega}_{21}^\ep )\big) 
\big((u_1^\ep -\omega_{11}^\ep )-( u_2^\ep - \omega_{21}^\ep )\big)dx \\
-\int_\Omega\div_x \mathcal{Q}(x,\xoe) \big((u_1^\ep-\omega_{11}^\ep)-( u_2^\ep- \omega_{21}^\ep)\big)
\big((\hat{u}_{1}^\ep  -\hat{\omega}_{11}^\ep )-(\hat{u}_{2}^\ep  -\hat{\omega}_{21}^\ep )\big)dx\\
\leq c \ep^{1\over 2} \big(||\hat{u}_1^\ep(t)-\hat{\omega}_{11}^\ep(t)||_{V} + ||\hat{u}_2^\ep(t) -\hat{\omega}_{21}^\ep(t) ||_{V}\big).\\
%\leq c \ep^{1\over 2} \big[ \big( ||u_1^\ep(t,x)-\omega_{11}^\ep(t,x)||_{V}^2 \big)^{1 \over 2}
%+\big(  ||u_2^\ep(t,x)-\omega_{21}^\ep(t,x)||_{V}^2 \big)\big].
%\leq c \ep^{1\over 2} \big(||\hat{u}_1^\ep(t) -\hat{\omega}_{11}^\ep(t) ||_{V} + ||\hat{u}_2^\ep(t) -\hat{\omega}_{21}^\ep (t)||_{V}\big).
\end{split}
\eeq
As $u_i^\ep (t,x) = \hat{u}_i^\ep(t,x) e^{\lambda t},\ \omega_{i1}^\ep (t,x)=\hat{\omega}_{i1}^\ep(t,x)e^{\lambda t} $
\beq
\label{eq:lamb2}
\bsp
 \int_\Omega {\mathcal C}_{11}^\epsilon {\partial   \over \partial t}  (\hat{u}_{1}^\ep  -\hat{\omega}_{11}^\ep ) (\hat{u}_{1}^\ep  -\hat{\omega}_{11}^\ep )  dx 
+\lambda \int_\Omega {\mathcal C}_{11}^\epsilon (\hat{u}_{1}^\ep  -\hat{\omega}_{11}^\ep )^2 dx 
+ \int_\Omega \kappa_1^\epsilon |\nabla \hat{u}_{1}^\ep -\nabla \hat{\omega}_{11}^\ep |^2dx \\
+ \int_\Omega {\mathcal C}_{22}^\epsilon {\partial   \over \partial t}  (\hat{u}_{2}^\ep  -\hat{\omega}_{21}^\ep ) (\hat{u}_{2}^\ep  -\hat{\omega}_{21}^\ep )  dx 
+\lambda \int_\Omega {\mathcal C}_{22}^\epsilon (\hat{u}_{2}^\ep  -\hat{\omega}_{21}^\ep )^2 dx 
+ \int_\Omega \kappa_2^\epsilon |\nabla \hat{u}_{2}^\ep -\nabla \hat{\omega}_{21}^\ep |^2dx \\
- 2\int_\Omega \mathcal{Q}(x,\xoe)\cdot\big((\nabla \hat{u}_1^\ep -\nabla\hat{\omega}_{11}^\ep)-(\nabla \hat{u}_2^\ep-\nabla \hat{\omega}_{21}^\ep)\big) 
\big((\hat{u}_{1}^\ep  -\hat{\omega}_{11}^\ep )-(\hat{u}_{2}^\ep  -\hat{\omega}_{21}^\ep ))\big) dx \\
-\int_\Omega \div_x \mathcal{Q}(x,\xoe) \big((\hat{u}_1^\ep-\hat{\omega}_{11}^\ep)-( \hat{u}_2^\ep- \hat{\omega}_{21}^\ep)\big)^2 dx\\
%\leq c \ep^{1\over 2} \big[ \big( ||u_1^\ep(t,x)-\omega_{11}^\ep(t,x)||_{V}^2 \big)^{1 \over 2}
%+\big(  ||u_2^\ep(t,x)-\omega_{21}^\ep(t,x)||_{V}^2 \big)\big].
\leq c e^{-\lambda t}\ep^{1\over 2} \big(||\hat{u}_1^\ep(t) -\hat{\omega}_{11}^\ep(t) ||_{V} + ||\hat{u}_2^\ep(t) -\hat{\omega}_{21}^\ep(t) ||_{V}\big).
\end{split}
\eeq

Integrating over $[0,T]$ we get
\beq
\label{eq:lamb3}
\bsp
\frac{1}{2} \int_\Omega {\mathcal C}_{11}^\epsilon  (\hat{u}_{1}^\ep(T)  -\hat{\omega}_{11}^\ep(T) )^2 dx dt
+\lambda \int_0^T \int_\Omega {\mathcal C}_{11}^\epsilon (\hat{u}_{1}^\ep  -\hat{\omega}_{11}^\ep )^2 dx dt
+  \int_0^T\int_\Omega \kappa_1^\epsilon |\nabla \hat{u}_{1}^\ep -\nabla \hat{\omega}_{11}^\ep |^2dxdt \\
+\frac{1}{2} \int_\Omega {\mathcal C}_{22}^\epsilon  (\hat{u}_{2}^\ep(T)  -\hat{\omega}_{21}^\ep(T) )^2 dx dt
+\lambda  \int_0^T\int_\Omega {\mathcal C}_{22}^\epsilon (\hat{u}_{2}^\ep  -\hat{\omega}_{21}^\ep )^2 dx dt
+  \int_0^T\int_\Omega \kappa_2^\epsilon |\nabla \hat{u}_{2}^\ep -\nabla \hat{\omega}_{21}^\ep |^2dxdt \\
-2 \int_0^T\int_\Omega \mathcal{Q}(x,\xoe)\cdot\big((\nabla \hat{u}_1^\ep -\nabla\hat{\omega}_{11}^\ep)-(\nabla \hat{u}_2^\ep-\nabla \hat{\omega}_{21}^\ep)\big) 
\big((\hat{u}_{1}^\ep  -\hat{\omega}_{11}^\ep )-(\hat{u}_{2}^\ep  -\hat{\omega}_{21}^\ep ))\big) dxdt \\
- \int_0^T\int_\Omega \div_x \mathcal{Q}(x,\xoe) \big((\hat{u}_1^\ep-\hat{\omega}_{11}^\ep)-( \hat{u}_2^\ep- \hat{\omega}_{21}^\ep)\big)^2 dx dt\\
%\leq c \ep^{1\over 2} \big[ \big( ||u_1^\ep(t,x)-\omega_{11}^\ep(t,x)||_{V}^2 \big)^{1 \over 2}
%+\big(  ||u_2^\ep(t,x)-\omega_{21}^\ep(t,x)||_{V}^2 \big)\big].
\leq c \ep^{1\over 2} \left( \int_0^T||\hat{u}_1^\ep(t) -\hat{\omega}_{11}^\ep (t)||_{V} + ||\hat{u}_2^\ep(t) -\hat{\omega}_{21}^\ep (t)||_{V}dt\right) 
+\frac{1}{2} \int_0^T \int_\Omega {\mathcal C}_{11}^\epsilon  (\hat{u}_{1}^\ep(0)  -\hat{\omega}_{11}^\ep(0) )^2 dxdt\\
+\frac{1}{2}  \int_0^T\int_\Omega {\mathcal C}_{22}^\epsilon  (\hat{u}_{2}^\ep(0)  -\hat{\omega}_{21}^\ep(0) )^2 dxdt
\end{split}
\eeq
By Cauchy Schwartz and Young's inequalities, using the boundedness of $\mathcal{Q}(x,y)$, one obtains 
\beq
\label{eq:lamb4}
\bsp
\frac{1}{2} \int_\Omega {\mathcal C}_{11}^\epsilon  (\hat{u}_{1}^\ep(T)  -\hat{\omega}_{11}^\ep(T) )^2 dx dt
+\lambda \int_0^T \int_\Omega {\mathcal C}_{11}^\epsilon (\hat{u}_{1}^\ep  -\hat{\omega}_{11}^\ep )^2 dx dt
+  \int_0^T\int_\Omega \kappa_1^\epsilon |\nabla \hat{u}_{1}^\ep -\nabla \hat{\omega}_{11}^\ep |^2dxdt \\
\frac{1}{2} \int_\Omega {\mathcal C}_{22}^\epsilon  (\hat{u}_{2}^\ep(T)  -\hat{\omega}_{21}^\ep(T) )^2 dx dt
+\lambda  \int_0^T\int_\Omega {\mathcal C}_{22}^\epsilon (\hat{u}_{2}^\ep  -\hat{\omega}_{21}^\ep )^2 dx dt
+  \int_0^T\int_\Omega \kappa_2^\epsilon |\nabla \hat{u}_{2}^\ep -\nabla \hat{\omega}_{21}^\ep |^2dxdt \\
{\displaystyle-c_0(||\nabla \hat{u}_1^\ep(t) -\nabla\hat{\omega}_{11}^\ep(t) ||^2_{{L^2(0,T;H)}}+||\nabla \hat{u}_2^\ep(t)  -\nabla\hat{\omega}_{21}^\ep(t) ||^2_{{L^2(0,T;H)}}) }\\
\displaystyle{- c_1(||\hat{u}_{1}^\ep(t)   -\hat{\omega}_{11}^\ep(t) ||^2_{{L^2(0,T;H)}}+||\hat{u}_{2}^\ep (t)  -\hat{\omega}_{21}^\ep(t) ||^2_{{L^2(0,T;H)}})}\\
\displaystyle{-c_2(||\hat{u}_{1}^\ep (t)  -\hat{\omega}_{11}^\ep(t) ||^2_{{L^2(0,T;H)}}+||\hat{u}_{2}^\ep (t)  -\hat{\omega}_{21}^\ep(t) ||^2_{{L^2(0,T;H)}})}\\
%\leq c \ep^{1\over 2} \big[ \big( ||u_1^\ep(t,x)-\omega_{11}^\ep(t,x)||_{V}^2 \big)^{1 \over 2}
%+\big(  ||u_2^\ep(t,x)-\omega_{21}^\ep(t,x)||_{V}^2 \big)\big].
\leq c \ep^{1\over 2} \bigg(\big( \int_0^T||\hat{u}_1^\ep -\hat{\omega}_{11}^\ep ||^2_{V} dt\big)^\frac{1}{2}+ \big(\int_0^T||\hat{u}_2^\ep(t)  -\hat{\omega}_{21}^\ep(t)  ||^2_{V}dt\big)^\frac{1}{2} \bigg)
+\frac{1}{2} \int_0^T \int_\Omega {\mathcal C}_{11}^\epsilon  (\hat{u}_{1}^\ep(0)  -\hat{\omega}_{11}^\ep(0) )^2 dxdt\\
+\frac{1}{2}  \int_0^T\int_\Omega {\mathcal C}_{22}^\epsilon  (\hat{u}_{2}^\ep(0)  -\hat{\omega}_{21}^\ep(0) )^2 dxdt
\end{split}
\eeq
where the constant $c_0>0$ can be chosen to be smaller than $\underline{\kappa}$ in \eqref{eq:coercivity}. Choosing $\lambda$ large enough, we obtain from (\ref{eq:coercivity}),
\beq
\bsp
c||\hat{u}_1^\ep(T) - \hat{\omega}_{11}^\ep(T) ||_H^2 + c||\hat{u}_2^\ep(T) - \hat{\omega}_{21}^\ep(T) ||_H^2 
+c||\nabla \hat{u}_1^\ep - \nabla\hat{\omega}_{11}^\ep||_{L^2(0,T;H)}^2 +c||\nabla \hat{u}_2^\ep - \nabla \hat{\omega}_{21}^\ep||_{L^2(0,T;H)}^2\\
\leq c \ep^{1\over 2} ||\nabla \hat{u}_1^\ep - \nabla \hat{\omega}_{11}^\ep||_{L^2(0,T;H)} 
+c \ep^{1\over 2} ||\nabla \hat{u}_2^\ep - \nabla \hat{\omega}_{21}^\ep||_{L^2(0,T;H)}
+c||\hat{u}_1^\ep(0) - \hat{\omega}_{11}^\ep(0) ||_H^2 + c||\hat{u}_2^\ep(0) -\hat{\omega}_{21}^\ep(0) ||_H^2.
\end{split}
\eeq
%Note that $e^{-\lambda T} \leq e^{-\lambda t} \leq 1 $. Then we have
Thus
\beq
\bsp
&||u_1^\ep(T) - \omega_{11}^\ep(T) ||_H^2 + ||u_2^\ep(T) - \omega_{21}^\ep(T) ||_H^2 
+||\nabla u_1^\ep - \nabla \omega_{11}^\ep||_{L^2(0,T;H)}^2 
+||\nabla u_2^\ep - \nabla \omega_{21}^\ep||_{L^2(0,T;H)}^2\\
&\leq c \ep^{1\over 2} ||\nabla u_1^\ep - \nabla \omega_{11}^\ep||_{L^2(0,T;H)} 
+c \ep^{1\over 2} ||\nabla u_2^\ep - \nabla \omega_{21}^\ep||_{L^2(0,T;H)}
+c||u_1^\ep(0) - \omega_{11}^\ep(0) ||_H^2 + c||u_2^\ep(0) - \omega_{21}^\ep(0) ||_H^2.
\end{split}
\eeq
Since $u_i^\ep(0) = u_{i0}(0) = g_i(x)$, we deduce that
\beq
u_i^\ep(0) - \omega_{i1}^\ep(0) = u_i^\ep(0) - u_{i0}(0) - \ep \tau u_{i1}(0,x,\xoe) = -\ep \tau u_{i1}(0,x,\xoe).
\eeq
As $u_{i1}(t,x,y) \in C([0,T] \times \bar{\Omega} \times \bar{Y})$, we have {$||u_i^\ep(0) - \omega_{i1}^\ep(0) ||_H \leq c \ep$}. 
From this we obtain
{
\beq
\bsp
||\nabla u_1^\ep - \nabla \omega_{11}^\ep||_{L^2(0,T;H)} +||\nabla u_2^\ep - \nabla \omega_{21}^\ep||_{L^2(0,T;H)}
\leq c \ep^{1\over 2}.
\end{split}
\eeq
From \eqref{eq:main132}, we have
\beq
\bsp
||\nabla u_1^\ep - \nabla u_{11}^\ep||_{L^2(0,T;H)} +||\nabla u_2^\ep - \nabla u_{21}^\ep||_{L^2(0,T;H)}
\leq c \ep^{1\over 2}.
\end{split}
\eeq
}
The conclusion follows.\hfill$\Box$

\section{Conclusions}

In this paper, we analyzed the homogenization of a two-scale dual-continuum system. 
%Given dual-continuum system consists of two equations 
The coupled exchange terms are scaled as $\mathcal{O}(\frac{1}{\ep})$. This scale gives an interesting homogenization limit which contains %we obtained a homogenized equation that has 
convection, coupled reaction terms with negative interaction coefficients while the original two scale system does not contain these features.
 We proved rigorously the homogenization convergence. We proved rigorously also the homogenization convergence rate. These proofs of homogenization convergence and error are significantly more complicated than those for the scaling $O(\frac{1}{\ep^2})$ considered in \cite{park2019hierarchical} due to %\vhh{the} $\frac{1}{\ep}$ scale of \vhh{the} exchange terms in the original equation and 
 the complicated form of the homogenized equation. 
\clearpage
\appendix
\section{Existence and uniqueness of weak solutions}
In this appendix, we present the proof of the existence and uniqueness of a weak solution of (\ref{eq:main1}) and (\ref{eq:main9}).
%Throughout this paper, we denote the spaces $H$ and $H_0^1(\Omega)$ as $H$ and $V$ respectively. 
In variational form, problem (\ref{eq:main1}) becomes :
Find $u_1^\ep, u_2^\ep \in L^2(0,T; V)$ such that $\frac{\partial u_1^\ep}{\partial t}, \frac{\partial u_2^\ep}{\partial t} \in L^2(0,T; V')$ and
\beq
\label{eq:var}
\bsp
\int_0^T \int_\Omega {\mathcal C}_{11}^\epsilon(x){\partial u_1^\epsilon(t,x) \over \partial t} \phi_1(t,x) & dx dt 
+ \int_0^T \int_\Omega \kappa_1^\epsilon(x)\nabla u_1^\epsilon(t,x) \cdot \nabla \phi_1(t,x) dx dt\\
&+  {1 \over \epsilon} \int_0^T \int_\Omega Q^\epsilon(x)(u_1^\epsilon(t,x)-u_2^\epsilon(t,x)) \phi_1(t,x)dx dt
=  \int_0^T \int_\Omega q \phi_1(t,x) dx dt\\
\int_0^T \int_\Omega {\mathcal C}_{22}^\epsilon(x){\partial u_2^\epsilon(t,x) \over \partial t} \phi_2(t,x) & dx dt 
+ \int_0^T \int_\Omega \kappa_2^\epsilon(x)\nabla u_2^\epsilon(t,x) \cdot \nabla \phi_2(t,x) dx dt\\
&+  {1 \over \epsilon} \int_0^T \int_\Omega Q^\epsilon(x)(u_2^\epsilon(t,x)-u_1^\epsilon(t,x)) \phi_2(t,x)dx dt
=  \int_0^T \int_\Omega q \phi_2(t,x) dx dt
\end{split}
\eeq
for all $\phi_1, \phi_2 \in L^2(0,T; V)$. The initial conditions are
$u_1^\ep(0,x) = g_1(x) \in H$, and $u_2^\ep(0,x) = g_2(x) \in H$. 
Let $W$ be the space $V\times V$.
We define a bilinear form $a : W \times W \to \mathbb{R}$ as
\beq
\bsp
&a((u_1(t), u_2(t)),(v_1(t),v_2(t)))\\
&=  \int_\Omega \kappa_1^\epsilon\nabla u_1(t)\cdot \nabla v_1(t) dx
+\int_\Omega \kappa_2^\epsilon\nabla u_2(t)\cdot \nabla v_2(t) dx
+  {1 \over \epsilon}  \int_\Omega Q^\epsilon(u_1(t)-u_2(t)) v_1(t)dx \\
&+  {1 \over \epsilon} \int_\Omega Q^\epsilon(u_2(t)-u_1(t)) v_2(t)dx. 
\end{split}
\eeq
We have the following theorem.
\begin{theorem}
\label{unibound}
Assume that the vector function $\mathcal{Q}(x,y)$ is in $ C^1(\bar{\Omega};C^1(\bar{Y}))^2$. Then the sequences $u_1^\epsilon$ and $u_2^\epsilon$ satisfying (\ref{eq:main1}) are uniformly bounded in $L^\infty(0,T;H)$ and $L^2(0,T;V)$.
\end{theorem}
\begin{proof}
As $\div \mathcal{Q}(x,\xoe) = \div_x \mathcal{Q}(x,\xoe) + \frac{1}{\ep} \div_y \mathcal{Q}(x,\xoe)$, 
$ Q(x,\xoe) = \ep \div \mathcal{Q}(x,\xoe) - \ep \div_x \mathcal{Q}(x,\xoe)$.
 Note that
  \begin{equation}
\label{eq:main66}
 \begin{split}
 \int_{\Omega} Q^\epsilon(x)(u_2^\epsilon-u_1^\epsilon)\phi_1 dx = -\ep \int_\Omega \mathcal{Q}(x,\xoe)\cdot \nabla(u_2^\epsilon-u_1^\epsilon) \phi_1  dx - \ep \int_\Omega \mathcal{Q}(x,\xoe) \cdot \nabla \phi_1 (u_2^\epsilon-u_1^\epsilon)   dx
\\ -\ep \int_\Omega\div_x \mathcal{Q}(x,\xoe)  (u_2^\epsilon-u_1^\epsilon)\phi_1   dx,
 \end{split}
 \end{equation}
 for all $\phi_1$, $\phi_2$ $\in {\mathcal C}_0^\infty(\Omega)$. Thus, from \eqref{eq:main65}, we have
 \begin{equation}
\label{eq:main67}
 \begin{split}
\int_0^T\int_{\Omega}{\mathcal C}_{11}^\epsilon{\partial u_1^\epsilon  \over \partial t}\phi_1  dx dt
+ \int_0^T\int_{\Omega} \kappa_1^\epsilon\nabla u_1^\epsilon \cdot \nabla \phi_1  dx dt
=-\int_0^T\int_\Omega \mathcal{Q}(x,\xoe)\cdot \nabla(u_2^\epsilon-u_1^\epsilon) \phi_1  dx dt 
-\\ \int_0^T\int_\Omega \mathcal{Q}(x,\xoe) \cdot \nabla \phi_1 (u_2^\epsilon-u_1^\epsilon)   dx dt
- \int_0^T\int_\Omega \div_x \mathcal{Q}(x,\xoe)  (u_2^\epsilon-u_1^\epsilon)\phi_1   dx dt 
+\int_0^T \int_{\Omega}q \phi_1  dx dt.\\
%\int_{\Omega}{\mathcal C}_{22}^\epsilon{\partial u_2^\epsilon  \over \partial t} \phi_2  dx+ \int_{\Omega} \kappa_2^\epsilon\nabla u_2^\epsilon \cdot \nabla \phi_2  dx
%-{1 \over \epsilon} \int_{\Omega} Q^\epsilon(u_1^\epsilon-u_2^\epsilon)\phi_2  dx = \int_{\Omega}q \phi_2  dx.
 \end{split}
 \end{equation}
 We let $\hat{u}_1^\ep = u_1^\ep e^{-\lambda t}$, $\hat{u}_2^\ep = u_2^\ep e^{-\lambda t}$ and $\check{\phi}_1 = \phi_1 e^{\lambda t}$. Then,
  \begin{equation}
\label{eq:main68'*'}
 \begin{split}
\int_0^T\int_{\Omega}{\mathcal C}_{11}^\epsilon{\partial \hat{u}_1^\epsilon  \over \partial t}\check{\phi}_1  dx dt
+\lambda  \int_0^T\int_{\Omega}{\mathcal C}_{11}^\epsilon \hat{u}_1^\epsilon\check{\phi}_1  dx dt
+ \int_0^T\int_{\Omega} \kappa_1^\epsilon\nabla \hat{u}_1^\epsilon \cdot \nabla \check{\phi}_1  dx dt\\
=-\int_0^T\int_\Omega \mathcal{Q}(x,\xoe)\cdot \nabla(\hat{u}_2^\epsilon-\hat{u}_1^\epsilon) \check{\phi}_1  dx dt 
- \int_0^T\int_\Omega \mathcal{Q}(x,\xoe) \cdot \nabla \check{\phi}_1 (\hat{u}_2^\epsilon-\hat{u}_1^\epsilon)   dx dt\\
- \int_0^T\int_\Omega \div_x\mathcal{Q}(x,\xoe)  (\hat{u}_2^\epsilon-\hat{u}_1^\epsilon)\check{\phi}_1   dx dt 
+\int_0^T \int_{\Omega}q \check{\phi}_1 e^{-\lambda t} dx dt.\\
%\int_{\Omega}{\mathcal C}_{22}^\epsilon{\partial u_2^\epsilon  \over \partial t} \phi_2  dx+ \int_{\Omega} \kappa_2^\epsilon\nabla u_2^\epsilon \cdot \nabla \phi_2  dx
%-{1 \over \epsilon} \int_{\Omega} Q^\epsilon(u_1^\epsilon-u_2^\epsilon)\phi_2  dx = \int_{\Omega}q \phi_2  dx.
 \end{split}
 \end{equation}
 Similarly, we have
   \begin{equation}
\label{eq:main68'**}
 \begin{split}
\int_0^T\int_{\Omega}{\mathcal C}_{22}^\epsilon{\partial \hat{u}_2^\epsilon  \over \partial t}\check{\phi}_2  dx dt
+\lambda  \int_0^T\int_{\Omega}{\mathcal C}_{22}^\epsilon \hat{u}_2^\epsilon\check{\phi}_2  dx dt
+ \int_0^T\int_{\Omega} \kappa_2^\epsilon\nabla \hat{u}_2^\epsilon \cdot \nabla \check{\phi}_2  dx dt\\
=-\int_0^T\int_\Omega \mathcal{Q}(x,\xoe)\cdot \nabla(\hat{u}_1^\epsilon-\hat{u}_2^\epsilon) \check{\phi}_2  dx dt 
- \int_0^T\int_\Omega \mathcal{Q}(x,\xoe) \cdot \nabla \check{\phi}_2 (\hat{u}_1^\epsilon-\hat{u}_2^\epsilon)   dx dt\\
- \int_0^T\int_\Omega \div_x \mathcal{Q}(x,\xoe)  (\hat{u}_1^\epsilon-\hat{u}_2^\epsilon)\check{\phi}_2  dx dt 
+\int_0^T \int_{\Omega}q \check{\phi}_2 e^{-\lambda t} dx dt.\\
%\int_{\Omega}{\mathcal C}_{22}^\epsilon{\partial u_2^\epsilon  \over \partial t} \phi_2  dx+ \int_{\Omega} \kappa_2^\epsilon\nabla u_2^\epsilon \cdot \nabla \phi_2  dx
%-{1 \over \epsilon} \int_{\Omega} Q^\epsilon(u_1^\epsilon-u_2^\epsilon)\phi_2  dx = \int_{\Omega}q \phi_2  dx.
 \end{split}
 \end{equation}
Let $\check{\phi}_1 =  \hat{u}_1^\ep$, $\check{\phi}_2 =  \hat{u}_2^\ep$. Taking the sum of the above two equations we get
  \begin{equation}
\label{eq:main67}
 \begin{split}
\int_0^T\int_{\Omega}{\mathcal C}_{11}^\epsilon{\partial \hat{u}_1^\epsilon  \over \partial t}\hat{u}_1^\ep  dx dt
+\int_0^T\int_{\Omega}{\mathcal C}_{22}^\epsilon{\partial \hat{u}_2^\epsilon  \over \partial t}\hat{u}_2^\ep  dx dt
+ \int_0^T\int_{\Omega} \kappa_1^\epsilon\nabla \hat{u}_1^\epsilon \cdot \nabla \hat{u}_1^\ep  dx dt
+ \int_0^T\int_{\Omega} \kappa_2^\epsilon\nabla \hat{u}_2^\epsilon \cdot \nabla \hat{u}_2^\ep dx dt\\
+\lambda  \int_0^T\int_{\Omega}{\mathcal C}_{11}^\epsilon (\hat{u}_1^\epsilon)^2  dx dt
+\lambda  \int_0^T\int_{\Omega}{\mathcal C}_{22}^\epsilon (\hat{u}_2^\epsilon)^2  dx dt\\
=2\int_0^T\int_\Omega \mathcal{Q}(x,\xoe)\cdot \nabla(\hat{u}_1^\epsilon-\hat{u}_2^\epsilon)(\hat{u}_1^\epsilon-\hat{u}_2^\epsilon)   dx dt 
+ \int_0^T\int_\Omega \div_x\mathcal{Q}(x,\xoe)  (\hat{u}_1^\epsilon-\hat{u}_2^\epsilon)^2  dx dt \\
+\int_0^T \int_{\Omega}q e^{-\lambda t} (\hat{u}_1^\epsilon+\hat{u}_2^\epsilon) dx dt.
%\int_{\Omega}{\mathcal C}_{22}^\epsilon{\partial u_2^\epsilon  \over \partial t} \phi_2  dx+ \int_{\Omega} \kappa_2^\epsilon\nabla u_2^\epsilon \cdot \nabla \phi_2  dx
%-{1 \over \epsilon} \int_{\Omega} Q^\epsilon(u_1^\epsilon-u_2^\epsilon)\phi_2  dx = \int_{\Omega}q \phi_2  dx.
 \end{split}
 \end{equation}
 Thus,
   \begin{equation}
\label{eq:main68}
 \begin{split}
\frac{1}{2}\int_{\Omega}{\mathcal C}_{11}^\epsilon |\hat{u}_1^\ep(T,x)|^2  dx
+\frac{1}{2}\int_{\Omega}{\mathcal C}_{22}^\epsilon|\hat{u}_2^\ep(T,x)|^2  dx
+ \int_0^T\int_{\Omega} \kappa_1^\epsilon\nabla \hat{u}_1^\epsilon \cdot \nabla \hat{u}_1^\ep  dx dt
+ \int_0^T\int_{\Omega} \kappa_2^\epsilon\nabla \hat{u}_2^\epsilon \cdot \nabla \hat{u}_2^\ep dx dt\\
+\lambda  \int_0^T\int_{\Omega}{\mathcal C}_{11}^\epsilon (\hat{u}_1^\epsilon)^2  dx dt
+\lambda  \int_0^T\int_{\Omega}{\mathcal C}_{22}^\epsilon (\hat{u}_2^\epsilon)^2  dx dt\\
=2\int_0^T\int_\Omega \mathcal{Q}(x,\xoe)\cdot \nabla(\hat{u}_1^\epsilon-\hat{u}_2^\epsilon)(\hat{u}_1^\epsilon-\hat{u}_2^\epsilon)   dx dt 
+ \int_0^T\int_\Omega \div_x\mathcal{Q}(x,\xoe)  (\hat{u}_1^\epsilon-\hat{u}_2^\epsilon)^2  dx dt \\
+\int_0^T \int_{\Omega}q e^{-\lambda t} (\hat{u}_1^\epsilon+\hat{u}_2^\epsilon) dx dt
+\frac{1}{2}\int_{\Omega}{\mathcal C}_{11}^\epsilon |\hat{u}_1^\ep(0,x)|^2  dx
+\frac{1}{2} \int_{\Omega}{\mathcal C}_{22}^\epsilon|\hat{u}_2^\ep(0,x)|^2  dx.
%\int_{\Omega}{\mathcal C}_{22}^\epsilon{\partial u_2^\epsilon  \over \partial t} \phi_2  dx+ \int_{\Omega} \kappa_2^\epsilon\nabla u_2^\epsilon \cdot \nabla \phi_2  dx
%-{1 \over \epsilon} \int_{\Omega} Q^\epsilon(u_1^\epsilon-u_2^\epsilon)\phi_2  dx = \int_{\Omega}q \phi_2  dx.
 \end{split}
 \end{equation}
 Since $\mathcal{Q}(x,y) \in C^1(\bar{\Omega};C^1(\bar{Y}))$, using Cauchy Schwartz and Young's inequalities, we have
 \begin{equation}
\label{eq:main69}
 \begin{split}
2\int_0^T\int_\Omega \mathcal{Q}(x,\xoe)\cdot \nabla&(\hat{u}_1^\epsilon-\hat{u}_2^\epsilon)(\hat{u}_1^\epsilon-\hat{u}_2^\epsilon)  dx dt \\
&\leq c_0 (||\hat{u}_1^\epsilon||_{L^2(0,T;V)}^2+||\hat{u}_2^\epsilon||_{L^2(0,T;V)}^2)+c_1(||\hat{u}_1^\epsilon||_{L^2(0,T;H)}^2+||\hat{u}_2^\epsilon||_{L^2(0,T;H)}^2).
%\int_{\Omega}{\mathcal C}_{22}^\epsilon{\partial u_2^\epsilon  \over \partial t} \phi_2  dx+ \int_{\Omega} \kappa_2^\epsilon\nabla u_2^\epsilon \cdot \nabla \phi_2  dx
%-{1 \over \epsilon} \int_{\Omega} Q^\epsilon(u_1^\epsilon-u_2^\epsilon)\phi_2  dx = \int_{\Omega}q \phi_2  dx.
 \end{split}
 \end{equation}
 Similarly,
  \begin{equation}
\label{eq:main70}
 \begin{split}
\int_0^T\int_\Omega \div_x \mathcal{Q}(x,\xoe)   (\hat{u}_1^\epsilon-\hat{u}_2^\epsilon)^2  dx dt 
\leq c_2 (||\hat{u}_1^\epsilon||_{L^2(0,T;H)}^2+||\hat{u}_2^\epsilon||_{L^2(0,T;H)}^2)
%+c_3(||\hat{u}_1^\epsilon||_{L^2((0,T);H)}^2+||\hat{u}_2^\epsilon||_{L^2((0,T);H)}^2).
%\int_{\Omega}{\mathcal C}_{22}^\epsilon{\partial u_2^\epsilon  \over \partial t} \phi_2  dx+ \int_{\Omega} \kappa_2^\epsilon\nabla u_2^\epsilon \cdot \nabla \phi_2  dx
%-{1 \over \epsilon} \int_{\Omega} Q^\epsilon(u_1^\epsilon-u_2^\epsilon)\phi_2  dx = \int_{\Omega}q \phi_2  dx.
 \end{split}
 \end{equation}
 and
   \begin{equation}
\label{eq:main71}
 \begin{split}
\int_0^T \int_{\Omega}q e^{-\lambda t} (\hat{u}_1^\epsilon+\hat{u}_2^\epsilon) dx dt 
\leq c_3+ c_4(||\hat{u}_1^\epsilon||_{L^2(0,T;H)}^2+||\hat{u}_2^\epsilon||_{L^2(0,T;H)}^2).
%\int_{\Omega}{\mathcal C}_{22}^\epsilon{\partial u_2^\epsilon  \over \partial t} \phi_2  dx+ \int_{\Omega} \kappa_2^\epsilon\nabla u_2^\epsilon \cdot \nabla \phi_2  dx
%-{1 \over \epsilon} \int_{\Omega} Q^\epsilon(u_1^\epsilon-u_2^\epsilon)\phi_2  dx = \int_{\Omega}q \phi_2  dx.
 \end{split}
 \end{equation}
Choosing $c_0$ sufficiently small and $\lambda$ sufficiently large, we deduce $\hat{u}_1^\epsilon$ and $\hat{u}_2^\epsilon$, thus, 
$u_1^\epsilon$ and $u_2^\epsilon$ are uniformly bounded in $L^\infty(0,T;H)$ and $L^2(0,T;V)$.
 \end{proof}

\begin{lemma}
\label{garding}
Assume $Q(x,y) \in L^\infty(\Omega\times Y)$ and $\kappa_i(x,y) \in L^\infty(\Omega\times Y)$.
There exists $C>0$ such that
\beq
\label{eq:bdd}
a((u_1, u_2),(v_1,v_2)) \leq C \big(||\nabla u_1 ||_H^2 +||\nabla u_2||_H^2 \big)^{\frac{1}{2}} \cdot \big(||\nabla v_1 ||_H^2 +|| \nabla v_2||_H^2 \big)^{\frac{1}{2}}
\eeq
for $(u_1, u_2),(v_1,v_2) \in W$.
And there exists $k \geq 0$ such that 
\beq
a((\phi_1,\phi_2),(\phi_1,\phi_2)) + k ||\phi_1||_H^2 + k ||\phi_2 ||_H^2 \geq \alpha (||\nabla \phi_1||^2_H + ||\nabla \phi_2||^2_H),
\eeq
for all $\phi_1, \phi_2 \in V$. Here, $C$ and $k$ depend on $\ep$.
\end{lemma}
\begin{proof}
It is not difficult to show (\ref{eq:bdd}).
Since $Q \in C(\bar\Omega\times\bar Y)$, we have
\beq
\bsp
a((u_1,u_2), (u_1,u_2))\\
=
\int_\Omega \kappa_1^\ep \nabla u_1\cdot \nabla u_1 dx
+ \int_\Omega \kappa_2^\ep \nabla u_2 \cdot\nabla u_2 dx
+ \frac{1}{\ep}\int_\Omega Q^\ep (u_1-u_2)^2 dx\\
%= 
%\int_\Omega \kappa_1^\ep \nabla u_1 \nabla u_1 dx
%+ \int_\Omega \kappa_2^\ep \nabla u_2 \nabla u_2 dx
%- 2 \int_\Omega \mathcal{Q}^\ep \cdot \nabla(u_1-u_2) (u_1-u_2) dx
%- \int_\Omega \div_x \mathcal{Q} (u_1-u_2)^2 dx\\
\geq
\underline{\kappa} (\norm{\nabla u_1}_{H}^2+\norm{\nabla u_2}_{H}^2)
- k( \norm{u_1}_{H}^2+\norm{u_2}_{H}^2)
\end{split}
\eeq
for some $k > 0$ depending on $\ep$. The last inequality follows from Cauchy-Schwarz and Young's inequalities.
\end{proof}

%Let us denote the spaces $H$ and $H_0^1(\Omega)$ as $H$ and $V$ respectively.
\begin{theorem}
\label{unique_original}
There exists a unique solution for problem (\ref{eq:var}).
\end{theorem}
\begin{proof}
We follow the standard proof for parabolic equations in \cite{wlokapartial}.
%Since $q \in H$, 
We note that $u_1^\ep, u_2^\ep$ are weak solutions of (\ref{eq:var}) if for almost all $t\in [0,T]$
\beq
\label{eq:vartt}
\bsp
\int_\Omega {\mathcal C}_{11}^\epsilon{\partial u_{1}^\ep(t) \over \partial t} \phi_1  dx 
+  \int_\Omega \kappa_1^\epsilon\nabla u_{1}^\ep(t) \cdot \nabla  \phi_1 dx 
+  {1 \over \epsilon} \int_\Omega Q^\epsilon(u_{1}^\ep(t)-u_{2}^\ep(t))  \phi_1 dx 
=   \int_\Omega q  \phi_1dx \\
 \int_\Omega {\mathcal C}_{22}^\epsilon{\partial u_{2}^\ep(t) \over \partial t} \phi_2 dx 
+ \int_\Omega \kappa_2^\epsilon\nabla u_{2}^\ep(t) \cdot \nabla \phi_2  dx 
+  {1 \over \epsilon} \int_\Omega Q^\epsilon(u_{2}^\ep(t)-u_{1}^\ep(t)) \phi_2  dx 
=  \int_\Omega q \phi_2  dx 
\end{split}
\eeq
for all $\phi_1, \phi_2 \in V$.
Let $\{\omega_{k}\}$ be an orthogonal basis of $V$ and an orthonormal basis of $H$.
For fixed integer $m > 0$, we consider functions
\beq
\bsp
u_{1m}^\ep(t) = \sum\limits_{k=1}^m d_{1m}^k(t) \omega_k, \ \
u_{2m}^\ep(t) = \sum\limits_{k=1}^m d_{2m}^k(t) \omega_k,
\end{split}
\eeq
where the coefficients $d_{1m}^k$, $d_{2m}^k$ satisfy
\beq
\bsp
d_{1m}^k(0) = \int_\Omega g_1 \omega_k dx,\ \
d_{2m}^k(0) = \int_\Omega g_2 \omega_k dx \\
\end{split}
\eeq
and
\beq
\label{eq:vart}
\bsp
\int_\Omega {\mathcal C}_{11}^\epsilon{\partial u_{1m}^\ep(t) \over \partial t} \omega_{k_1}  dx 
+  \int_\Omega \kappa_1^\epsilon\nabla u_{1m}^\ep(t)\cdot \nabla \omega_{k_1} dx 
+  {1 \over \epsilon} \int_\Omega Q^\epsilon(u_{1m}^\ep(t)-u_{2m}^\ep(t)) \omega_{k_1} dx \\
=   \int_\Omega q \omega_{k_1}dx \\
 \int_\Omega {\mathcal C}_{22}^\epsilon{\partial u_{2m}^\ep(t) \over \partial t}\omega_{k_2}  dx 
+ \int_\Omega \kappa_2^\epsilon\nabla u_{2m}^\ep(t) \cdot\nabla \omega_{k_2} dx 
+  {1 \over \epsilon} \int_\Omega Q^\epsilon(u_{2m}^\ep(t)-u_{1m}^\ep(t)) \omega_{k_2} dx \\
=  \int_\Omega q \omega_{k_2} dx 
\end{split}
\eeq
a.e. on [0,T], where $k_1,k_2 = 1,2,\dots, m$.
This problem can be written as a system of ODEs 
\beq
\label{eq:ODE}
\bsp
 \sum_{l=1}^m  [M_{1}]_{kl} \frac{d}{dt} d_{1m}^k(t) + \sum_{l=1}^m [A_1+M_{Q}]_{kl} d_{1m}^l(t) - \sum_{l=1}^m [M_Q]_{kl} d_{2m}^l (t)= \int q\omega_k dx\\
 \sum_{l=1}^m [M_{2}]_{kl} \frac{d}{dt} d_{2m}^k(t) + \sum_{l=1}^m [A_2+M_{Q}]_{kl} d_{2m}^l(t) - \sum_{l=1}^m [M_Q]_{kl} d_{1m}^l (t)= \int a\omega_k dx
\end{split}
\eeq
for $k = 1, 2, \dots , m$, where 
\beq
[M_{i}]_{kl} = \int_\Omega C_{ii}^\ep \omega_k \omega_l dx, \
[M_{Q}]_{kl} = {1 \over \ep}\int_\Omega Q^\ep \omega_k \omega_l dx,\
[A_i]_{kl} = \int_\Omega \kappa_i^\ep \nabla \omega_k\cdot \nabla \omega_l dx.
\eeq
Since $M_{1}$ and $M_{2}$ are positive definite and symmetric Gram matrices, they are invertible. Hence, (\ref{eq:ODE}) has unique solutions.
% and $u_{1m}^\ep, u_{2m}^\ep$ that satisfy (\ref{eq:vart}) are uniquely determined for each $m$. 

It can be shown that $u_{1m}^\ep, u_{2m}^\ep$ are uniformly bounded in both $L^2(0,T;V),\, L^\infty(0,T;H)$ and
$ {\mathcal C}_{11}^\epsilon{\partial u_{1m}^\ep \over \partial t}, {\mathcal C}_{22}^\epsilon{\partial u_{2m}^\ep \over \partial t}$ 
are uniformly bounded in $L^2(0,T;V')$ for all $m$.  The proof is similar to that of \ref{unibound} and \ref{weakconv}.

From these results, we deduce that there exist functions $u_1^\ep, u_2^\ep$, $\eta_1^\ep, \eta_2^\ep$ such that
\beq
\bsp
u_{im}^\ep \rightharpoonup u_i^\ep \ \text{in}  \ L^2(0,T;V),\ \ \
 {\mathcal C}_{ii}^\epsilon{\partial u_{im}^\ep\over \partial t}
 \rightharpoonup \ \eta_i^\ep \ \text{in}  \ L^2(0,T;V'), \ \ i = 1, 2.
\end{split}
\eeq
Let $\psi_1(t), \psi_2(t) \in C^1[0,T]$ with $\psi_1(T)=\psi_2(T) = 0$. Let $\phi_{1k} = \psi_1 \omega_{k}$, $\phi_{2k} = \psi_2 \omega_{k}$. From (\ref{eq:vart}) 
 %by $\phi_1(t), \phi_2(t)$ and integrate over $(0,T)$ and 
 we get
\beq
\label{eq:vart1}
\bsp
\int_0^T \int_\Omega {\mathcal C}_{11}^\epsilon{\partial u_{1m}^\ep \over \partial t} \phi_{1k}  dx dt
+ \int_0^T \int_\Omega \kappa_1^\epsilon\nabla u_{1m}^\ep\cdot \nabla \phi_{1k}  dx dt
+  {1 \over \epsilon}\int_0^T \int_\Omega Q^\epsilon(u_{1m}^\ep-u_{2m}^\ep) \phi_{1k}  dx dt\\
=  \int_0^T \int_\Omega q \phi_{1k} dx dt\\
\int_0^T \int_\Omega {\mathcal C}_{22}^\epsilon{\partial u_{2m}^\ep \over \partial t}\phi_{2k}   dx dt
+ \int_0^T\int_\Omega \kappa_2^\epsilon\nabla u_{2m}^\ep\cdot \nabla \phi_{2k} dx dt
+  {1 \over \epsilon} \int_0^T\int_\Omega Q^\epsilon(u_{2m}^\ep-u_{1m}^\ep) \phi_{2k}  dxdt \\
= \int_0^T \int_\Omega q \phi_{2k}  dx dt.
\end{split}
\eeq
%where $\phi_{1k} = \phi_1 \omega_{k1}$, $\phi_{2k} = \phi_2 \omega_{k2}$.
Since $\phi(T) = 0$, integrating by parts we obtain
\beq
\label{eq:uniq1}
\bsp
-\int_0^T \int_\Omega {\mathcal C}_{11}^\epsilon u_{1m}^\ep \frac{\partial\phi_{1k} }{\partial t} dx dt
+ \int_0^T \int_\Omega \kappa_1^\epsilon\nabla u_{1m}^\ep \cdot\nabla \phi_{1k}  dx dt
+  {1 \over \epsilon}\int_0^T \int_\Omega Q^\epsilon(u_{1m}^\ep-u_{2m}^\ep) \phi_{1k}  dx dt\\
=  \int_0^T \int_\Omega q \phi_{1k} dx dt + \int_\Omega {\mathcal C}_{11}^\ep u_{1m}^\ep(0) \phi_{1k}(0) dx \\
-\int_0^T \int_\Omega {\mathcal C}_{22}^\epsilon u_{2m}^\ep \frac{\partial\phi_{2k} }{\partial t} dx dt
+ \int_0^T\int_\Omega \kappa_2^\epsilon\nabla u_{2m}^\ep \cdot\nabla \phi_{2k} dx dt
+  {1 \over \epsilon} \int_0^T\int_\Omega Q^\epsilon(u_{2m}^\ep-u_{1m}^\ep) \phi_{2k}  dxdt \\
= \int_0^T \int_\Omega q \phi_{2k}  dx dt 
+ \int_\Omega {\mathcal C}_{22}^\ep u_{2m}^\ep(0) \phi_{2k}(0) dx.
\end{split}
\eeq
Note that $u_{1m}^\ep(0) \to g_1$, $u_{2m}^\ep(0) \to g_2$ in $H$ as $m \to \infty$. Passing to the limit, $m \to \infty$, we obtain 
\beq
\label{eq:uniq2}
\bsp
-\int_0^T \int_\Omega {\mathcal C}_{11}^\epsilon u_{1}^\ep \frac{\partial\phi_{1k} }{\partial t} dx dt
+ \int_0^T \int_\Omega \kappa_1^\epsilon\nabla u_{1}^\ep\cdot \nabla \phi_{1k}  dx dt
+  {1 \over \epsilon}\int_0^T \int_\Omega Q^\epsilon(u_{1}^\ep-u_{2}^\ep) \phi_{1k}  dx dt\\
=  \int_0^T \int_\Omega q \phi_{1k} dx dt 
+ \int_\Omega {\mathcal C}_{11}^\ep g_1 \phi_{1k}(0) dx, \\
-\int_0^T \int_\Omega {\mathcal C}_{22}^\epsilon u_{2}^\ep \frac{\partial\phi_{2k} }{\partial t} dx dt
+ \int_0^T\int_\Omega \kappa_2^\epsilon\nabla u_{2}^\ep\cdot \nabla \phi_{2k} dx dt
+  {1 \over \epsilon} \int_0^T\int_\Omega Q^\epsilon(u_{2}^\ep-u_{1}^\ep) \phi_{2k}  dxdt \\
= \int_0^T \int_\Omega q \phi_{2k}  dx dt 
+ \int_\Omega {\mathcal C}_{22}^\ep g_2 \phi_{2k}(0) dx.
\end{split}
\eeq
{
We partially integrate the first terms of the equations in (\ref{eq:uniq2}) and obtain
\beq
\label{eq:uniq3}
\bsp
\int_0^T \int_\Omega {\mathcal C}_{11}^\epsilon \frac{\partial u_{1}^\ep }{\partial t}\phi_{1k} dx dt
+ \int_\Omega  {\mathcal C}_{11}^\epsilon u_1^\ep(0) \phi_{1k}(0) dx
+ \int_0^T \int_\Omega \kappa_1^\epsilon\nabla u_{1}^\ep\cdot \nabla \phi_{1k}  dx dt\\
+  {1 \over \epsilon}\int_0^T \int_\Omega Q^\epsilon(u_{1}^\ep-u_{2}^\ep) \phi_{1k}  dx dt
=  \int_0^T \int_\Omega q \phi_{1k} dx dt 
+ \int_\Omega {\mathcal C}_{11}^\ep g_1 \phi_{1k}(0) dx, \\
\int_0^T \int_\Omega {\mathcal C}_{22}^\epsilon \frac{\partial u_{2}^\ep }{\partial t}\phi_{2k} dx dt
+ \int_\Omega  {\mathcal C}_{22}^\epsilon u_2^\ep(0) \phi_{2k}(0) dx
+ \int_0^T\int_\Omega \kappa_2^\epsilon\nabla u_{2}^\ep\cdot \nabla \phi_{2k} dx dt\\
+  {1 \over \epsilon} \int_0^T\int_\Omega Q^\epsilon(u_{2}^\ep-u_{1}^\ep) \phi_{2k}  dxdt 
= \int_0^T \int_\Omega q \phi_{2k}  dx dt 
+ \int_\Omega {\mathcal C}_{22}^\ep g_2 \phi_{2k}(0) dx.
\end{split}
\eeq
As this holds for all $\psi_1, \psi_2 \in {\mathcal D}((0,T))$, it follows that
\beq
\label{eq:vart2}
\bsp
\int_\Omega {\mathcal C}_{11}^\epsilon{\partial u_{1}^\ep(t) \over \partial t} \omega_{k}  dx 
+  \int_\Omega \kappa_1^\epsilon\nabla u_{1}^\ep(t) \cdot\nabla \omega_{k} dx 
+  {1 \over \epsilon} \int_\Omega Q^\epsilon(u_{1}^\ep(t)-u_{2}^\ep(t))\omega_{k} dx 
=   \int_\Omega q \omega_{k}dx, \\
 \int_\Omega {\mathcal C}_{22}^\epsilon{\partial u_{2}^\ep(t) \over \partial t}\omega_{k}  dx 
+ \int_\Omega \kappa_2^\epsilon\nabla u_{2}^\ep(t)\cdot \nabla \omega_{k} dx 
+  {1 \over \epsilon} \int_\Omega Q^\epsilon(u_{2}^\ep(t)-u_{1}^\ep(t)) \omega_{k} dx 
=  \int_\Omega q \omega_{k} dx 
\end{split}
\eeq
a.e. on [0,T], and
%From this and \eqref{eq:uniq3}, one obtains
\beq
\label{eq:init1}
\bsp
 \int_\Omega  {\mathcal C}_{11}^\epsilon u_1^\ep(0) \omega_{k}(0) dx
 = \int_\Omega {\mathcal C}_{11}^\ep g_1 \omega_{k}(0) dx, \
  \int_\Omega  {\mathcal C}_{22}^\epsilon u_2^\ep(0) \omega_{k}(0) dx
 = \int_\Omega {\mathcal C}_{22}^\ep g_2 \omega_{k}(0) dx
\end{split}
\eeq
for all $k$. Thus, from (\ref{eq:vart2}) and (\ref{eq:init1}), we deduce
\beq
\label{eq:vart3}
\bsp
\int_\Omega {\mathcal C}_{11}^\epsilon{\partial u_{1}^\ep(t) \over \partial t} \phi_1  dx 
+  \int_\Omega \kappa_1^\epsilon\nabla u_{1}^\ep(t)\cdot \nabla \phi_1 dx 
+  {1 \over \epsilon} \int_\Omega Q^\epsilon(u_{1}^\ep(t)-u_{2}^\ep(t)) \phi_1 dx 
=   \int_\Omega q\phi_1dx, \\
 \int_\Omega {\mathcal C}_{22}^\epsilon{\partial u_{2}^\ep(t) \over \partial t}\phi_2  dx 
+ \int_\Omega \kappa_2^\epsilon\nabla u_{2}^\ep(t) \cdot\nabla \phi_2dx 
+  {1 \over \epsilon} \int_\Omega Q^\epsilon(u_{2}^\ep(t)-u_1^\ep(t)) \phi_2 dx 
=  \int_\Omega q \phi_2 dx 
\end{split}
\eeq
a.e. on [0,T], for all $\phi_1, \phi_2 \in V$
}
%%%%%%%%%%%%%%%%%%%%%%%%   commented below %%%%%%%%%%%%%%%%%%%%%%%%%%%
\begin{comment}
Since $u_i^\ep \in L^2(0,T;V)$ and $q \in L^2((0,T)\times\Omega)$, we have ${\mathcal C}_{ii}^\epsilon{\partial u_{i}^\ep(t)\over \partial t}  \in L^2(0,T;V')$.
Letting $\phi_1 = \phi_{1k}(t)$, $\phi_2 = \phi_{2k}(t)$ in (\ref{eq:vart3}), where $ \phi_{1k}$, $ \phi_{2k}$ are defined as above, integrating over $[0,T]$, we have
\beq
\label{eq:uniq3}
\bsp
-\int_0^T \int_\Omega {\mathcal C}_{11}^\epsilon u_{1}^\ep \frac{\partial\phi_{1k} }{\partial t} dx dt
+ \int_0^T \int_\Omega \kappa_1^\epsilon\nabla u_{1}^\ep \cdot\nabla \phi_{1k}  dx dt
+  {1 \over \epsilon}\int_0^T \int_\Omega Q^\epsilon(u_{1}^\ep-u_{2}^\ep) \phi_{1k}  dx dt\\
=  \int_0^T \int_\Omega q \phi_{1k} dx dt 
+ \int_\Omega {\mathcal C}_{11}^\ep u_1^\ep(0) \phi_{1k}(0) dx \\
-\int_0^T \int_\Omega {\mathcal C}_{22}^\epsilon u_{2}^\ep \frac{\partial\phi_{2k} }{\partial t} dx dt
+ \int_0^T\int_\Omega \kappa_2^\epsilon\nabla u_{2}^\ep \cdot\nabla \phi_{2k} dx dt
+  {1 \over \epsilon} \int_0^T\int_\Omega Q^\epsilon(u_{2}^\ep-u_{1}^\ep) \phi_{2k}  dxdt \\
= \int_0^T \int_\Omega q \phi_{2k}  dx dt 
+ \int_\Omega {\mathcal C}_{22}^\ep u_2^\ep(0) \phi_{2k}(0) dx.
\end{split}
\eeq
for all $k_i \in \mathbb{N}$. 
\end{comment}
%%%%%%%%%%%%%%%%%%%%%%%%%% commented above %%%%%%%%%%%%%%%%%%%%%%%%%%%%%%%%
{and ${\mathcal C}_{ii}^\ep u_i^\ep(0) = {\mathcal C}_{ii}^\ep g_i$, hence, $u_i^\ep(0) = g_i$.}
Thus, $u_1^\ep, u_2^\ep$ are solutions of (\ref{eq:var}).
We now show the uniqueness of the solutions.
Assume $u_1^\ep, u_2^\ep$, $v_1^\ep, v_2^\ep$ are two solution sets of (\ref{eq:var}).
We let $  u_1^\ep-v_1^\ep=\delta_1$, $ u_2^\ep-v_2^\ep = \delta_2$. Then from (\ref{eq:var}), we get
%%%%

\beq
\bsp
\int_0^T \int_\Omega {\mathcal C}_{11}^\epsilon(x){\partial \delta_1(t,x) \over \partial t} \phi_1(t,x) dx dt 
&+ \int_0^T \int_\Omega \kappa_1^\epsilon(x)\nabla \delta_1(t,x) \cdot\nabla \phi_1(t,x) dx dt\\
&+  {1 \over \epsilon} \int_0^T \int_\Omega Q^\epsilon(x)(\delta_1 (t,x)-\delta_2(t,x)) \phi_1(t,x)dx dt
=  0,\\
\int_0^T \int_\Omega {\mathcal C}_{22}^\epsilon(x){\partial \delta_2(t,x) \over \partial t} \phi_2(t,x) dx dt 
&+ \int_0^T \int_\Omega \kappa_2^\epsilon(x)\nabla \delta_2(t,x)\cdot \nabla \phi_2(t,x) dx dt\\
&+  {1 \over \epsilon} \int_0^T \int_\Omega Q^\epsilon(x)(\delta_2(t,x)-\delta_1(t,x)) \phi_2(t,x)dx dt
=  0
\end{split}
\eeq
for all $\phi_1, \phi_2 \in L^2(0,T;V)$.
Letting $\hat{\delta}_1(t) = \delta_1(t) e^{-\lambda t}$, $\hat{\delta}_2(t) = \delta_2(t) e^{-\lambda t}$, 
$\check{\phi}_1 = \phi_1 e^{\lambda t}$ and $\check{\phi}_2 = \phi_2 e^{\lambda t}$,
 we have
\beq
\bsp
\int_0^T \int_\Omega {\mathcal C}_{11}^\epsilon{\partial \hat{\delta}_1\over \partial t} \check{\phi}_1 dx dt
+ \int_0^T \int_\Omega {\mathcal C}_{22}^\epsilon{\partial \hat{\delta}_2 \over \partial t} \check{\phi}_2dx dt 
+\lambda \int_0^T \int_\Omega {\mathcal C}_{11}^\ep \hat{\delta}_1 \check{\phi}_1 dx dt
+\lambda \int_0^T \int_\Omega {\mathcal C}_{22}^\ep \hat{\delta}_2 \check{\phi}_2 dx dt\\
+\int_0^T a((\hat{\delta}_1(t),\hat{\delta}_2(t)),(\check{\phi}_1(t),\check{\phi}_2(t))) dt=0.
\end{split}
\eeq
Letting $\check{\phi}_i = \hat{\delta}_i$, we have
\beq
\bsp
\frac{1}{2} \int_\Omega {\mathcal C}_{11}^\epsilon|\hat{\delta}_1(T)|^2 dx
+\frac{1}{2} \int_\Omega {\mathcal C}_{22}^\epsilon|\hat{\delta}_2(T)|^2 dx
+\lambda \int_0^T \int_\Omega {\mathcal C}_{11}^\ep |\hat{\delta}_1|^2 dx dt
+\lambda \int_0^T \int_\Omega {\mathcal C}_{22}^\ep |\hat{\delta}_2|^2 dx dt\\
+\int_0^T a((\hat{\delta}_1(t),\hat{\delta}_2(t)),(\hat{\delta}_1(t),\hat{\delta}_2(t))) dt=0.
\end{split}
\eeq
Note that $\hat{\delta}_i(0) = 0$.
By Lemma \ref{garding}, choosing sufficiently large $\lambda$, we have
\beq
\bsp
\frac{1}{2} \int_\Omega {\mathcal C}_{11}^\epsilon|\hat{\delta}_1(T)|^2 dx
+\frac{1}{2} \int_\Omega {\mathcal C}_{22}^\epsilon|\hat{\delta}_2(T)|^2 dx
+\alpha \int_0^T (||\nabla \hat{\delta}_1(t)||^2_H + ||\nabla \hat{\delta}_2(t)||^2_H) dt\\
\leq
\frac{1}{2} \int_\Omega {\mathcal C}_{11}^\epsilon|\hat{\delta}_1(T)|^2 dx
+\frac{1}{2} \int_\Omega {\mathcal C}_{22}^\epsilon|\hat{\delta}_2(T)|^2 dx
+\lambda \int_0^T \int_\Omega {\mathcal C}_{11}^\ep |\hat{\delta}_1|^2 dx dt
+\lambda \int_0^T \int_\Omega {\mathcal C}_{22}^\ep |\hat{\delta}_2|^2 dx dt\\
+\int_0^T a((\hat{\delta}_1(t),\hat{\delta}_2(t)),(\hat{\delta}_1(t),\hat{\delta}_2(t))) dt =0.
\end{split}
\eeq
%Thus for some $c>0$, we have
%\beq
%\bsp
%\frac{1}{2} \int_\Omega {\mathcal C}_{11}^\epsilon|\hat{\delta}_1(T)|^2 dx
%+\frac{1}{2} \int_\Omega {\mathcal C}_{22}^\epsilon|\hat{\delta}_2(T)|^2 dx
%+c \int_0^T (||\hat{\delta}_1(t)||^2_V + ||\hat{\delta}_2(t)||^2_V) dt
%\leq 0.
%\end{split}
%\eeq
This implies $\hat{\delta}_1 = \hat{\delta}_2 = 0$, thus, $\delta_1 = \delta_2 = 0$ a.e. on $[0,T]\times \Omega$.
We deduce $u_1^\ep = v_1^\ep, \ u_2^\ep = v_2^\ep$.

\end{proof}

Now we show the uniqueness of a solution of the homogenized system \eqref{eq:main9}.
The homogenized problem (\ref{eq:main9}) can be written in variational form. We find $u_{10}, u_{20} \in L^2(0,T; V)$ such that $\frac{\partial u_{10}}{\partial t}, \frac{\partial u_{20}}{\partial t} \in L^2(0,T; V')$ satisfying
\begin{equation}
\label{eq:var_hom}
\begin{split}
\int^T_0 \int _\Omega\left( \int_Y {\mathcal C}_{11} dy\right){\partial u_{10} \over \partial t} \phi_1 dx dt
+ \int^T_0 \int _\Omega \kappa_1^*\nabla u_{10} \cdot\nabla \phi_1 dx dt
+\int^T_0 \int _\Omega (\int_Y \kappa_1 \nabla_y M_1  dy)\cdot\nabla \phi_1 (u_{20}-u_{10}) dx dt
\\ + \int^T_0 \int _\Omega [(\int_Y QN^i_1  dy)\frac{\partial u_{10}}{\partial x_i} - (\int_Y QN^i_2  dy)\frac{\partial u_{20}}{\partial x_i}] \phi_1 dx dt
- \int^T_0 \int _\Omega(\int_YQ(M_1+M_2)  dy)(u_{10}-u_{20}) \phi_1 dx dt 
\\= \int^T_0 \int _\Omega q \phi_1 dx dt,\\
\int^T_0 \int _\Omega \left(\int_Y {\mathcal C}_{22} dy\right){\partial u_{20} \over \partial t} \phi_2 dx dt
+ \int^T_0 \int _\Omega \kappa_2^*\nabla u_{20} \cdot \nabla \phi_2 dx dt
+\int^T_0 \int _\Omega (\int_Y \kappa_2 \nabla_y M_2  dy)\cdot \nabla \phi_2 (u_{10}-u_{20}) dx dt
\\ + \int^T_0 \int _\Omega [(\int_Y QN^i_2  dy)\frac{\partial u_{20}}{\partial x_i} - (\int_Y QN^i_1  dy)\frac{\partial u_{10}}{\partial x_i}] \phi_2 dx dt
- \int^T_0 \int _\Omega(\int_YQ(M_1+M_2)  dy)(u_{20}-u_{10}) \phi_2 dx dt 
\\= \int^T_0 \int _\Omega q \phi_2 dx dt,
\end{split}
\end{equation} 
for all $\phi_1, \phi_2 \in L^2(0,T;V)$.
We define the bilinear form $b : W \times W \to \mathbb{R}$ by
\beq
\bsp
b((u_{10}(t),u_{20}(t)),(\phi_1(t), \phi_2(t))) =\\
\int _\Omega \kappa_1^*\nabla u_{10} \cdot\nabla \phi_1 dx 
+ \int _\Omega (\int_Y \kappa_1 \nabla_y M_1  dy)\cdot \nabla \phi_1 (u_{20}-u_{10})dx
\\ + \int _\Omega [(\int_Y QN^i_1  dy)\frac{\partial u_{10}}{\partial x_i} - (\int_Y QN^i_2  dy)\frac{\partial u_{20}}{\partial x_i}] \phi_1 dx 
-  \int _\Omega(\int_YQ(M_1+M_2)  dy)(u_{10}-u_{20}) \phi_1 dx\\
+  \int _\Omega \kappa_2^*\nabla u_{20}\cdot \nabla \phi_2 dx 
+ \int _\Omega (\int_Y \kappa_2 \nabla_y M_2  dy)\cdot \nabla \phi_2 (u_{10}-u_{20}) dx 
\\ +  \int _\Omega [(\int_Y QN^i_2  dy)\frac{\partial u_{20}}{\partial x_i} - (\int_Y QN^i_1  dy)\frac{\partial u_{10}}{\partial x_i}] \phi_2 dx
-\int _\Omega(\int_YQ(M_1+M_2)  dy)(u_{20}-u_{10}) \phi_2 dx.
\end{split}
\eeq
\begin{lemma}
\label{garding_hom}
Assume $Q, \kappa_j \in C(\bar{\Omega};C(\bar{Y}))$, $ N^i_j ,\ M_j \in C(\bar{\Omega};C^1(\bar{Y}))$ for $j = 1,\ 2$.
There exists $C>0$ such that
\beq
\label{eq:b1}
b((u_1, u_2),(v_1,v_2)) \leq C \big(||\nabla u_1 ||_H^2 +||\nabla u_2||_H^2 \big)^{\frac{1}{2}} \cdot \big(||\nabla v_1 ||_H^2 +|| \nabla v_2||_H^2 \big)^{\frac{1}{2}}
\eeq
for $(u_1, u_2),(v_1,v_2) \in W$.
There exists $k \geq 0$ such that 
\beq
\label{eq:garding_hom}
b((u_1,u_2),(u_1,u_2)) + k ||u_1||_H^2 + k ||u_2 ||_H^2 \geq \alpha (||\nabla u_1||^2_H + ||\nabla u_2||^2_H)
\eeq
for all $u_1, u_2 \in V$, for a constant $\alpha>0$.
\end{lemma}
\begin{proof}
%\todo{need to be proved.}
We first show \eqref{eq:b1}. We have
\beq
\bsp
b((u_1, u_2),(v_1,v_2))\\
\leq
c (||\nabla u_1 ||_H \cdot ||\nabla v_1||_H +   ||\nabla v_1 ||_H \cdot ||u_1||_H +||\nabla v_1 ||_H \cdot ||u_2||_H
 +  ||\nabla u_1 ||_H \cdot || v_1||_H\\
+ ||\nabla u_2 ||_H \cdot || v_1||_H 
+  || u_1 ||_H \cdot || v_1||_H+   || u_2 ||_H \cdot || v_1||_H
+ ||\nabla u_2 ||_H \cdot ||\nabla v_2||_H +  ||\nabla v_2 ||_H \cdot ||u_2||_H \\
+||\nabla v_2 ||_H \cdot ||u_1||_H
+   ||\nabla u_2 ||_H \cdot || v_2||_H
+ ||\nabla u_1 ||_H \cdot || v_2||_H 
+   || u_2 ||_H \cdot || v_2||_H+  || u_2 ||_H \cdot || v_2||_H)
\\
\leq
c \big(||\nabla u_1 ||_H^2 +||\nabla u_2||_H^2 +|| u_1 ||_H^2 +|| u_2||_H^2\big)^{\frac{1}{2}} 
\cdot \big(||\nabla v_1 ||_H^2 +|| \nabla v_2||_H^2+ || v_1 ||_H^2 +|| v_2||_H^2\big)^{\frac{1}{2}}\\
\leq
 C \big(||\nabla u_1 ||_H^2 +||\nabla u_2||_H^2 \big)^{\frac{1}{2}} \cdot \big(||\nabla v_1 ||_H^2 +|| \nabla v_2||_H^2 \big)^{\frac{1}{2}}.
\end{split}
\eeq
The last inequality follows from Poincare inequality. 
We now prove (\ref{eq:garding_hom}). As $\kappa_1^*$ and $\kappa_2^*$ are positive definite, we have
\beq
\bsp
b((u_1, u_2),(u_1,u_2))\\
\geq
c_1 (||\nabla u_1||_H^2+||\nabla u_2||_H^2)
- c_2 ( ||\nabla u_1 ||_H \cdot ||u_1||_H +||\nabla u_1 ||_H \cdot ||u_2||_H
 +  ||\nabla u_1 ||_H \cdot || u_1||_H\\
+ ||\nabla u_2 ||_H \cdot || u_1||_H 
+  || u_1 ||_H \cdot || u_1||_H+   || u_2 ||_H \cdot || u_1||_H
 +  ||\nabla u_2 ||_H \cdot ||u_2||_H 
+||\nabla u_2 ||_H \cdot ||u_1||_H\\
+   ||\nabla u_2 ||_H \cdot || u_2||_H
+ ||\nabla u_1 ||_H \cdot || u_2||_H 
+   || u_2 ||_H \cdot || u_2||_H+  || u_2 ||_H \cdot || u_2||_H)\\
\geq
c_1 (||\nabla u_1||_H^2+||\nabla u_2||_H^2)
-  ( \varepsilon_1||\nabla u_1 ||_H^2 +\delta_1 ||u_1||_H^2 + \varepsilon_2||\nabla u_2||_H^2 +\delta_2 ||u_2||_H^2).\\
\end{split}
\eeq
Choosing $\varepsilon_1, \ \varepsilon_2$ small enough, we get the conclusion.
%we have $c_1 - \varepsilon_i > 0$ and the result follows.
\end{proof}
%Let us denote the spaces $H$ and $H_0^1(\Omega)$ as $H$ and $V$ respectively.
\begin{theorem}
\label{unique_hom}
There exists a unique solution for problem (\ref{eq:var_hom}).
\end{theorem}
\begin{proof} The existence follows from Theorem  \ref{weakconv}. We only prove the uniqueness.
Assume $(u_{10}, u_{20})$, $(v_{10}, v_{20})$ are two solutions of (\ref{eq:var_hom}).
We let $  u_{10}-v_{10}=\delta_1$, $ u_{20}-v_{20} = \delta_2$. From (\ref{eq:var_hom}), we obtain
\begin{equation}
\begin{split}
\int^T_0 \int _\Omega \int_Y {\mathcal C}_{11} dy{\partial \delta_1 \over \partial t} \phi_1 dx dt
+ \int^T_0 \int _\Omega \kappa_1^*\nabla \delta_1 \cdot\nabla \phi_1 dx dt
+\int^T_0 \int _\Omega (\int_Y \kappa_1 \nabla_y M_1  dy)\cdot\nabla \phi_1 (\delta_2-\delta_1) dx dt
\\ + \int^T_0 \int _\Omega [(\int_Y QN^i_1  dy)\frac{\partial \delta_1}{\partial x_i} - (\int_Y QN^i_2  dy)\frac{\partial \delta_2}{\partial x_i}] \phi_1 dx dt
- \int^T_0 \int _\Omega(\int_YQ(M_1+M_2)  dy)(\delta_1-\delta_2) \phi_1 dx dt 
\\=0,\\
\int^T_0 \int _\Omega \int_Y {\mathcal C}_{22} dy{\partial \delta_2 \over \partial t} \phi_2 dx dt
+ \int^T_0 \int _\Omega \kappa_2^*\nabla \delta_2 \cdot \nabla \phi_2 dx dt
+\int^T_0 \int _\Omega (\int_Y \kappa_2 \nabla_y M_2  dy)\cdot \nabla \phi_2 (\delta_1-\delta_2) dx dt
\\ + \int^T_0 \int _\Omega [(\int_Y QN^i_2  dy)\frac{\partial \delta_2}{\partial x_i} - (\int_Y QN^i_1  dy)\frac{\partial \delta_1}{\partial x_i}] \phi_2 dx dt
- \int^T_0 \int _\Omega(\int_YQ(M_1+M_2)  dy)(\delta_2-\delta_1) \phi_2 dx dt 
\\=0
\end{split}
\end{equation} 
for all $\phi_1, \phi_2 \in L^2(0,T;V)$.
Let $\hat{\delta}_1(t) = \delta_1(t) e^{-\lambda t}$, $\hat{\delta}_2(t) = \delta_2(t) e^{-\lambda t}$, 
$\check{\phi}_1 = \phi_1 e^{\lambda t}$ and $\check{\phi}_2 = \phi_2 e^{\lambda t}$. We have
\begin{equation}
\begin{split}
\int^T_0 \int _\Omega \int_Y {\mathcal C}_{11} dy{\partial \hat{\delta}_1 \over \partial t} \check{\phi}_1 dx dt
+\lambda \int^T_0 \int _\Omega \int_Y {\mathcal C}_{11} dy \hat{\delta}_1\check{\phi}_1 dx dt
+ \int^T_0 \int _\Omega \kappa_1^*\nabla \hat{\delta}_1 \cdot\nabla \check{\phi}_1 dx dt\\
+\int^T_0 \int _\Omega (\int_Y \kappa_1 \nabla_y M_1  dy)\cdot\nabla \check{\phi}_1 (\hat{\delta}_2-\hat{\delta}_1) dx dt
\\ + \int^T_0 \int _\Omega [(\int_Y QN^i_1  dy)\frac{\partial \hat{\delta}_1}{\partial x_i} - (\int_Y QN^i_2  dy)\frac{\partial \hat{\delta}_2}{\partial x_i}] \check{\phi}_1 dx dt
- \int^T_0 \int _\Omega(\int_YQ(M_1+M_2)  dy)(\hat{\delta}_1-\hat{\delta}_2) \check{\phi}_1 dx dt 
\\=0,\\
\int^T_0 \int _\Omega \int_Y {\mathcal C}_{22} dy{\partial \hat{\delta}_2 \over \partial t} \check{\phi}_2 dx dt
+\lambda \int^T_0 \int _\Omega \int_Y {\mathcal C}_{22} dy \hat{\delta}_2\check{\phi}_2dx dt
+ \int^T_0 \int _\Omega \kappa_2^*\nabla \hat{\delta}_2 \cdot \nabla \check{\phi}_2 dx dt\\
+\int^T_0 \int _\Omega (\int_Y \kappa_2 \nabla_y M_2  dy)\cdot \nabla \check{\phi}_2 (\hat{\delta}_1-\hat{\delta}_2) dx dt
\\ + \int^T_0 \int _\Omega [(\int_Y QN^i_2  dy)\frac{\partial \hat{\delta}_2}{\partial x_i} - (\int_Y QN^i_1  dy)\frac{\partial \hat{\delta}_1}{\partial x_i}] \check{\phi}_2 dx dt
- \int^T_0 \int _\Omega(\int_YQ(M_1+M_2)  dy)(\hat{\delta}_2-\hat{\delta}_1) \check{\phi}_2 dx dt 
\\=0.
\end{split}
\end{equation} 
We let $\check{\phi}_i = \hat{\delta}_i$.
Since $\hat{\delta}_i(0) = 0$, adding above 2 equations, we have
\beq
\bsp
\frac{1}{2} \int_\Omega \int_Y {\mathcal C}_{11} dy|\hat{\delta}_1(T)|^2 dx
+\frac{1}{2} \int_\Omega \int_Y {\mathcal C}_{22} dy|\hat{\delta}_2(T)|^2 dx\\
+\lambda \int_0^T \int_\Omega \int_Y {\mathcal C}_{11} dy |\hat{\delta}_1|^2 dx dt
+\lambda \int_0^T \int_\Omega\int_Y {\mathcal C}_{22} dy |\hat{\delta}_2|^2 dx dt
+\int_0^T b((\hat{\delta}_1(t),\hat{\delta}_2(t)),(\hat{\delta}_1(t),\hat{\delta}_2(t))) dt=0.
\end{split}
\eeq 
Choosing $\lambda$ large enough, we have
\beq
\bsp
\frac{1}{2} \int_\Omega \int_Y {\mathcal C}_{11} dy|\hat{\delta}_1(T)|^2 dx
+\frac{1}{2} \int_\Omega \int_Y {\mathcal C}_{22} dy|\hat{\delta}_2(T)|^2 dx
+ \alpha \int_0^T( ||\nabla \hat{\delta}_1(t)||_H^2 + ||\nabla \hat{\delta}_2(t)||_H^2) dt =0,
\end{split}
\eeq 
by Lemma \ref{garding_hom}. 
%For some $c>0$, we have
%\beq
%\bsp
%\frac{1}{2} \int_\Omega \int_Y {\mathcal C}_{11} dy|\hat{\delta}_1(T)|^2 dx
%+\frac{1}{2} \int_\Omega \int_Y {\mathcal C}_{22} dy|\hat{\delta}_2(T)|^2 dx
%+ c \int_0^T( || \hat{\delta}_1(t)||_V^2 + || \hat{\delta}_2(t)||_V^2) dt =0.
%\end{split}
%\eeq 
We deduce $\hat{\delta}_1 = \hat{\delta}_2 = 0$ thus $\delta_1 = \delta_2 = 0$. 
Thus, we conclude $u_{10} = v_{10}$, $u_{20} = v_{20}$. 
\end{proof}
\textbf{Acknowledgment}
%This material is based upon work supported by the National Science Foundation and National Research Foundation Singapore under Grant No. 1713805. 
{A part of this work is conducted when Jun Sur Richard Park was a visiting PhD student at Nanyang Technological University (NTU) under East Asia and Pacific Summer Institutes (EAPSI) programme organized by the US National Science Foundation (NSF) and Singapore National Research Foundation (NRF) under Grant No. 1713805. 
Jun Sur Richard Park thanks US NSF and Singapore NRF for the financial support and NTU for hospitality.}  Viet Ha Hoang is supported by the MOE Tier 2 grant MOE2017-T2-2-144

\bibliographystyle{siam}
\bibliography{referencesHomMulti}

\end{document}